\documentclass[final, reqno, 11pt]{amsart}
\usepackage{amsmath}
\usepackage{amsfonts}
\usepackage{amssymb}
\usepackage{amsthm}

\newtheorem{theorem}{Theorem}[section]
\newtheorem{lemma}[theorem]{Lemma}
\newtheorem{proposition}[theorem]{Proposition}
\newtheorem{corollary}[theorem]{Corollary}
\theoremstyle{remark}
\newtheorem{observation}[theorem]{Remark}

\newtheorem{definition}{Definition}[section]
\newcommand{\set}{\mathbb}
\newcommand{\R}{\mathbb R}

\newcommand{\U}{U}
\newcommand{\CC}{\Gamma}
\newcommand{\B}{\mc L}
\newcommand{\BV}{BV}

\newcommand{\dl}{\nabla}

\newcommand{\les}{\lesssim}
\newcommand{\mc}{\mathcal}
\newcommand{\be}{\begin{equation}}
\newcommand{\ee}{\end{equation}}
\newcommand{\bee}{\begin{align}}
\newcommand{\eee}{\end{align}}

\newcommand{\ba}{\begin{array}}
\newcommand{\ds}{\displaystyle}

\newcommand{\ea}{\end{array}}
\newcommand{\bpm}{\begin{pmatrix}}
\newcommand{\epm}{\end{pmatrix}}
\newcommand{\lb}{\label}
\DeclareMathOperator{\sgn}{sgn}
\DeclareMathOperator*{\slim}{s-lim}

\DeclareMathOperator{\Imim}{Im}

\newcommand{\dd}{{\,}{d}}

\renewcommand{\Im}{\Imim}

\newcommand{\Dil}{\mathrm{Dil}}

\title{The Schr\"{o}dinger Equation with a Potential in Rough Motion}
\author{Marius Beceanu}
\address{110 Frelinghuysen Rd., Rutgers Math.\ Dept., Piscataway, NJ 08854, USA}
\email{mbeceanu@rci.rutgers.edu}
\author{Avy Soffer}
\address{110 Frelinghuysen Rd., Rutgers Math.\ Dept., Piscataway, NJ 08854, USA}
\email{soffer@math.rutgers.edu}
\begin{document}
\maketitle
\numberwithin{equation}{section}
\begin{abstract}
This paper proves endpoint Strichartz estimates for the linear Schr\"{o}dinger equation in $\R^3$, with a time-dependent potential that keeps a constant profile and is subject to a rough motion, which need not be differentiable and may be large in norm. The potential is also subjected to a time-dependent rescaling, with a non-differentiable dilation parameter.\\
We use the Strichartz estimates to prove the non-dispersion of bound states, when the path is small in norm, as well as boundedness of energy.\\
We also include a sample nonlinear application of the linear results.
\end{abstract}

\section{Introduction}
\subsection{Overview}
Consider the linear Schr\"{o}dinger equation in $\set R^3$ with a
time dependent potential
\be\lb{1.1}
i\partial_t Z + H(t) Z = F,\ Z(0)= Z_0 \text{ given}, \ee where
\be\lb{eq_3.2} H(t) = H_0 + V(x,t) = -\Delta + G_{a(t)}V(x).
\ee
$G_{a}$ stands for an element of the Galilean group on $\set R^3$,
indexed by the~vector~$a$:
$$
a=(\gamma,\beta,v);\ G_a V(x) = e^{ivx} \Dil_{3/2}(\beta) V(x-\gamma).
$$
$\gamma:[0, \infty) \to \set R^3$ is a continuous curve
corresponding to translation, $\beta$ corresponds to rescaling, and $v$ to boost.
$\Dil_{3/2}$ stands for an element of a dilation group, see below.

Such Hamiltonians appear naturally in many applications, from
physical models, like charge transfer Hamiltonians and particle
models of Markovian potentials, to the mathematical analysis of
dispersive PDEs (e.g.\ stability of solitons). Time-dependent rotations can also be included (see \cite{bec}) and are relevant in the treatment of vortex-type dynamics. More general time-dependent symmetries can also be accounted for.

More substantial is the generalization to the multicenter case, where several such potentials with different time-dependent Galilean group actions are considered together. This case will be treated elsewhere.

We show that, for a general rough class of trajectories
$a(t)$, the basic dispersive and scattering estimates hold for
$H(t)$. The potential function $V(x)$ is assumed to be in a natural
$L^p$ space (i.e.\ $L^{3/2} \cap L^2$), with no smoothness or size assumptions.

The natural space that we find for the trajectories $a(t)$ is
\begin{definition}
$\CC = (\dot H^{1/2} \cap C) + \BV$.
\end{definition}

Here, $\BV$ stands for the space of functions of bounded variation.

A key property we use is that the space of distributional
derivatives of these trajectories is a Banach subspace of the dual of $\CC$: $\CC' \subset \CC^*$. The basic definitions of $\CC$, its dual, etc., and some of their fundamental properties are discussed in detail in the next section. It is shown in particular that $\CC$ is a Banach algebra under pointwise multiplication; see Lemma \ref{lemma24}.

In Section \ref{sect_3}, we estimate the integral operator arising
from the Duhamel formula representation of the solution of the
Schr\"odinger equation. The new ingredient is that the
kernel of the free propagator is conjugated by time dependent
Galilean group elements with rough time dependent parameters.

We use modulation equations to handle the bound states of the system, then estimate the solutions of the ensuing ODEs in a space of
distributions. Related results concerning ordinary differential equations in rough spaces have been obtained, for example, by Lyons \cite{lyons}, but we find that such methods do not apply directly to the current problem.

We prove the Strichartz estimates by a bootstrap argument involving the modulation equations, then control, uniformly in time, the $L^2$ and the $H^1$ norms of the solution in terms of the corresponding norms of the initial data.

Our approach allows, in particular, handling the modulation
equations in cases where there is no integrability in time (which is
replaced here by the $\dot H^{1/2} \cap C$ condition). It also applies to
Hamiltonians with self-similar potential, as they appear in the
study of blowup phenomena for NLS.

\subsection{Main result}
The linear Schr\"{o}dinger equation has symmetries corresponding to changes of position and velocity of its coordinate frame and to dilation. This allows us to accommodate three kinds of transformations of the potential $V$, i.e.\ translations, boost, and rescaling.

The rescaling of $V$ is dictated by the presence of $-\Delta$ in the equation, hence must be of the form $V \mapsto e^{\beta(x\dl+2)} V$. The rescaling of the solution $Z$, on the other hand, is dictated by the fact that we are considering $L^2$ solutions, so must be the $L^2$-unitary dilation $Z(x) \mapsto e^{\beta(x\dl+3/2)} Z(x)$.

Accordingly, let
$$
\gamma(t) = D(t) + 2\int_0^t v(s) \dd s
$$
and
\be\lb{isom}
S_{3/2}(t) = e^{\beta(t)(x\dl+2)} e^{\gamma(t)\dl},\ S(t) = e^{\beta(t)(x\dl+3/2)} e^{\gamma(t) \dl} e^{v(t) x}.
\ee
The main result of the paper is then the following:
\begin{proposition}[Strichartz estimates for time-dependent potentials]\lb{prop23} Consider the equation in $\R^3 \times [0, \infty)$
\be
i \partial_t Z + H(t) Z = F,\ Z(0) \text{ given},\ H(t) := -\Delta + S_{3/2}(t)^{-1} V,
\ee
with a real-valued, scalar potential
$$
S_{3/2}(t)^{-1} V(x) := e^{-\beta(t)(x\dl+2)} V(x-\gamma(t))
$$
of variable scale $\beta(t)$, driven by a curve
$$
\gamma(t) = D(t) + 2\int_0^t v(s) \dd s.
$$
Assume that $\ds\lim_{T \to \infty} \sup_{t \geq T} |\beta(t)-\beta(T)| << 1$, likewise for $v$ and $D$, and that $e^{\beta} \dot v, e^{-\beta} \dot D \in \dot H^{-1/2} \cap \partial_t^{-1} C$, $\beta \in \dot H^{1/2} \cap C$. Assume that $V \in L^{3/2} \cap L^2$ 
and that $-\Delta + V$ has no embedded eigenvalues or threshold resonances. 

Then there exist families of unitary operators $B(t) \in W^{1, \infty}_t$ and $A(t) \in \CC_t$ on $P_p L^2$ such that
\be\begin{aligned}
&\|P_c S(t) Z(t)\|_{L^{\infty}_t L^2_x \cap L^2_t L^{6, 2}_x} + \|B(t)^{-1} A(t) P_p S(t) Z(t)\|_{\CC} \les \\
&\les \|Z(0)\|_2 + \|P_c S(t) F(t)\|_{L^1_t L^2_x + L^2_t L^{6/5, 2}_x} + \|B(t)^{-1} A(t) P_p S(t) F(t)\|_{\CC'}.
\end{aligned}\ee
Moreover, for any $F \in \dot H^{1/2}_t \langle \dl \rangle^{-1} \langle x \rangle^{-1-\epsilon} L^2_x$ and any family of $L^2$-isometries $\tilde S(t)$ defined by (\ref{isom}) with $e^{\tilde \beta} \dot {\tilde v}, e^{-\tilde \beta} \dot {\tilde D} \in \dot H^{-1/2} \cap \partial_t^{-1} C$, $\tilde \beta \in \dot H^{1/2} \cap C$,
$$\begin{aligned}
\big\|\big\langle Z(t), \tilde S(t) F(t) \big\rangle\big\|_{\dot H^{1/2} \cap C} &\les \|Z(0)\|_2 + \|P_c S(t) F(t)\|_{L^1_t L^2_x + L^2_t L^{6/5, 2}_x} + \\
&+ \|B(t)^{-1} A(t) P_p S(t) F(t)\|_{\CC'}.
\end{aligned}$$
\end{proposition}
The same conclusion holds if $v$, $D$, $\beta \in \BV$ have sufficiently small jumps --- or if $v$, $D$, $\beta \in \CC$ have sufficiently small jumps and if all bound states of $-\Delta+V$ have three derivatives --- the latter guaranteed when $V$ has one derivative.

The proof becomes simpler if there is at most one bound state, with a further simplification if there are no bound states: then $v$, $D$, and $\beta$ can just be taken to have finite $L^{\infty}$ norms, if we also assume that they have locally small variation, i.e. at every $T$ there exists $\epsilon$ such that
$$
\sup_{t \in [T-\epsilon, T+\epsilon]} |v(t_1) - v(t_2)| << 1 \text{ and } \lim_{T \to \infty} \sup_{t \geq T} |v(t) - v(T)| << 1,
$$
and it also suffices to take $V \in L^{3/2, \infty}_0$ (the weak-$L^{3/2}$ closure of the set of bounded, compactly-supported functions). In this case, the proof essentially reduces to Lemma \ref{lemma2.1}.

Otherwise, the extra half-derivative is needed to control the interactions between the bound and dispersive states. The extra assumption on $V$, combined with the absence of threshold resonances and eigenstates, guarantees enough regularity and decay of the bound states of $-\Delta+V$ (by Lemma \ref{lema36}) to apply Lemma \ref{lema35}.

With bound states, the conclusion still holds true when $\beta$, $v$, $D \in \CC$, if they have locally small variation, as above --- for extra regularity for~$V$.

The proof also appears to work in even greater generality --- for $\beta$, $v$, and $D$ that are piecewise in $\CC$. However,  if for a partition of $[0, \infty)$
$$
[0, \infty) = [0, t_1] \cup [t_1, t_2] \cup \ldots \cup [t_{N-1}, \infty)
$$
$\chi_{[t_j, t_{j+1}]}(t) v(t) \in \CC([t_j, t_{j+1}])$, $0 \leq j \leq N-1$, then $v \in \CC$ --- so this is not a true generalization.

\subsection{Incomplete ionization and energy boundedness}
Strichartz estimates are shown to imply that ionization is controlled by the $\CC$ norm of the path $a(t)$, respectively that the energy stays bounded for all times. In particular, we show that the energy remains bounded for paths
$a(t)$ in $\dot H^{1/2} \cap C$, that the endpoint Strichartz estimates
hold, and that the $L^2$ wave operators exist and are asymptotically complete.

For sufficiently small perturbations of the trajectory in the $\dot H^{1/2} \cap C$ or $\CC$ norm, bound states never vanish entirely. Moreover, we can estimate the mass transfer between $P_c L^2$ and $P_p L^2$ due to their coupling through $\gamma$.

It thus follows that the ionization probability of such quantum
systems is bounded by the $\dot H^{1/2} \cap C$ norm of the path, so, for
small norm, there is no complete ionization:

\begin{corollary}[Incomplete ionization]\lb{ionization}
Let $Z$ be a solution to
$$\lb{11'}
i\partial_t Z + H(t) Z = 0,\ Z(0)= Z_0 \text{ given},
$$
with the potential $S_{3/2}(t)^{-1} V(x) := e^{-\beta(t)(x\dl+2)} V(x-\gamma(t))$ driven by a curve $\gamma(t)=D(t)+2\int_0^t v(s) \dd s$ and of variable scale $\beta(t)$ such that
$$
e^{\beta} \dot v,\ e^{-\beta} \dot D \in \dot H^{-1/2} \cap \partial_t^{-1} C,\ \beta \in \dot H^{1/2} \cap C.
$$
Assume that $V \in L^{3/2} \cap L^2$ is real and that $H=-\Delta + V$ admits bound states, but has no threshold eigenstates or resonances.

Then there exist operators $U_{cc}(t):P_c L^2 \to P_c L^2$, $U_{cp}(t): P_c L^2 \to P_p L^2$, $U_{pc}: P_p L^2 \to P_c L^2$, $U_{pp}(t): P_p L^2 \to P_p L^2$, such that
$$
\bpm P_c Z(x-\gamma(t)) \\ P_p Z(x-\gamma(t)) \epm = U(t) \bpm P_c Z(0) \\ P_p Z(0) \epm,\ U(t) = \bpm U_{cc}(t) & U_{pc}(t) \\ U_{cp}(t) & U_{pp}(t) \epm
$$
is a unitary transformation. Furthermore,
$$
\sup_t \|U_{cp}(t)\|_{\mc L(L^2_x, L^2_x)} + \|U_{pc}(t)\|_{\mc L(L^2_x, L^2_x)} \les \|\gamma\|_{\dot H^{1/2} \cap C}.
$$
If $P_p Z(0) \ne 0$ and $\|\gamma\|_{\dot H^{1/2} \cap C}$ is sufficiently small --- more specifically if $\|\gamma\|_{\dot H^{1/2} \cap C} \les \max(1, \frac {\|P_p Z(0)\|_2}{\|P_c Z(0)\|_2})$ --- then
\be
\liminf_{t \to \infty} \|P_p(t) Z(t)\|_{L^2_x} > 0.
\ee
As $t$ goes to $+\infty$, the wave operator defined by
$$
W_+ Z(0) := I -i\int_0^{\infty} e^{-itH_0} V(x-\gamma(t)) P_c Z(t) \dd t
$$
exists and the strong limit
\be\lb{channel}
W_{p+} Z(0):= \slim_{t \to \infty} B(t)^{-1} S(t) P_p Z(t)
\ee
defines a final (oscillating) state for the negative energy part of the solution.
\end{corollary}
$W_{p+}$ defined by (\ref{channel}) is also called the channel wave operator corresponding to bound states.

In other words, as $\|\gamma\|_{\dot H^{1/2} \cap C}$ gets smaller, the proportion of mass transferred between $P_c L^2$ and $P_p L^2$ goes to zero.


Contrary to this, if the path is taken to be Brownian, then
ionization results with probability one. This is shown in a separate
work \cite{BecS}; also see \cite{pillet} and \cite{cherem}, \cite{cherem2}. Thus, in some sense, paths of finite $\dot H^{1/2} \cap C$ norm are the
optimal, borderline case for the afore\-men\-tioned results.

Brownian motion is almost surely not in $\dot H^{1/2} \cap C$, but is always continuous and fails to be in $\dot H^{1/2}$ only logarithmically. Indeed, locally in time (e.g.\ for $t \in [0, 1]$) $B_t$ is almost surely in the Besov space $B^{1/2}_{2, \infty}$ and in $H^s$, $s<1/2$.

Thus, our results establish a threshold between the case where only a limited proportion of the mass can be transfered and the one where unlimited mass transfer can occur.

Next, let energy be defined as the sum of kinetic and potential energy,~viz.
$$
E[Z] = E_c[Z] + E_p[Z] = \frac 1 2 \langle -\Delta Z, Z \rangle + \langle V Z, Z \rangle
$$
and also consider the $L^2$ ``mass'' $M[Z] := \|Z\|_2^2$.
Note that $E[Z] = E[P_c Z] + E[P_p Z]$ and that
$$
E[Z](t):= E[Z(x+\gamma(t), t)]
$$
obeys the following conservation law:
\be\lb{econs}
\partial_t E[Z(x+\gamma(t), t)] = -\dot \gamma \langle Z, \dl V(x-\gamma(t)) Z \rangle.
\ee
Mass, on the other hand, is conserved under the time evolution (\ref{11'}).

Since $V \in L^{3/2, \infty}_0$, by writing $V=V_1 +V_2$, where $V_1 \in L^{\infty}$ and $\|V_2\|_{L^{3/2, \infty}}$ is small, one has that
\be\lb{control}
\|Z\|_{L^{6, 2}}^2 \les E[Z] + M[Z],
\ee
so
$$
|E_c[Z]| + |E_p[Z]| \les E[Z] + M[Z].
$$

Clearly, if $\gamma \in BV$ and $\dl V \in L^{\infty}$, energy remains uniformly bounded by (\ref{econs}). The same is true even if $\dl V \in L^{3/2, \infty}$, since then $\langle Z, \dl V(x-\gamma(t)) Z \rangle$ is controlled by (\ref{control}) and, by Gronwall's inequality, for all $t \geq 0$
$$
E[Z](t) \les e^{\|\gamma\|_{BV}} \big(\|Z(0)\|_{H^1} + \|\gamma\|_{BV} M[Z]\big).
$$

However, energy boundedness also holds under weaker conditions on $\gamma$, of the kind we assume in this paper (i.e.\ for nondifferentiable paths):
\begin{theorem}[Energy boundedness]\lb{energy} Let $Z$ solve
$$
i\partial_t Z + H(t) Z = 0,\ Z(0)= Z_0 \text{ given}.
$$
Assume that $V \in L^{3/2} \cap L^2$ is such that $H=-\Delta+V$ has no threshold eigenstate or resonance and that $\dl V \in L^{3/2, \infty}$.

Further assume that $V$ has a scale $\beta(t)$ and moves along $\gamma(t)=D(t)+2\int_0^t v(s) \dd s$, with $e^{\beta} \dot v, e^{-\beta} \dot D \in \dot H^{-1/2} \cap \partial_t^{-1} C$, $\beta \in \dot H^{1/2} \cap C$.
Then
$$
E[Z(x+\gamma(t), t)] \les \|Z(0)\|_{H^1}^2
$$
and this bound also holds for the kinetic and potential energies separately.
\end{theorem}

In the absence of bound states for $H = -\Delta+V$, this result can be further improved, as above, in the sense that we could take $\beta$, $v$, $D \in L^{\infty}$, of locally small variation, and $V \in L^{3/2, \infty}_0$. This is due to the fact that the proof of energy boundedness is based on Strichartz estimates, which still hold under these relaxed conditions.

\subsection{History of the problem} A rich literature exists concerning the Schr\"{o}\-din\-ger equation with moving and time-dependent potentials.

In the case of random motion, Pillet \cite{pillet} --- and Cheremshantsev \cite{cherem} \cite{cherem2} in the specific case of Brownian motion --- showed the existence and asymptotic completeness of the $L^2$ wave operators $\slim_{t \to \infty} e^{itH} e^{-itH_0}$. Thus, random motion of this kind will destroy the bound states of $-\Delta+V$, a behaviour opposite to the one we find for $\CC$ paths.

Other results (i.e.\ kernel decay or Strichartz estimates) for Schr\"{o}dinger's equation, in the time-periodic or quasiperiodic case, were obtained by Goldberg \cite{gol}, Costin--Lebowitz--Tanveer \cite{clt}, Galtbayar--Jensen--Yajima \cite{gjy}, Bourgain \cite{bou2}, and Wang \cite{wang}.

Results concerning time-dependent potentials of different kinds also belong to Howland \cite{how}, Kitada--Yajima \cite{kiya}, Rodnianski--Schlag \cite{rodsch}, and to Bourgain \cite{bou1}, \cite{bou3}, \cite{bou4}.

For linear in time trajectories, the optimal $L^p$ decay estimates,
which imply the Strichartz estimates, were proven in the multicenter
case (the charge transfer problem) by Rodnianski--Schlag--Soffer
\cite{rod2}.

\cite{bec}, in a context similar to that of the current paper, assumed that one can control the negative energy (bound-state) component of the solution $P_p(t) Z(t)$ in the Strichartz norm and used this to prove Strichartz estimates for the whole solution $Z(t)$. In particular, this approach worked if $H=-\Delta+V$ had no bound states at all.

By contrast, in this paper we do not assume any control of the bound states, but derive it explicitly instead, by solving the modulation equations in the rough function space $\Gamma$. For reasons explained above, this space has optimal regularity, as far as the results of this paper are concerned.

\subsection{An application} An interesting consequence is the following:

Consider a model of a heavy classical particle, with trajectory
given by a path $\gamma(t)$, going through a quantum fluid
described by NLS.

Then, the corresponding nonlinear equation for the
effective wave function satisfies the following NLS:
\be\lb{NLS}
i\frac{\partial \psi}{\partial t}=H(t)\psi + F(|\psi|)\psi,
\ee
where $H(t)$ is of the type we consider in this paper:
$$
H(t)=-\Delta+ V(x-\gamma(t)),\ \gamma(t) = D(t) + 2 \int_0^t v(s) \dd s.
$$
Absent the motion of the potential, this problem has two conserved quantities, ``mass'' ($M[\psi] := \|\psi\|_2^2$) and energy:
$$
E[\psi] := \frac 1 2  \int_{\R^3} |\dl \psi(x, t)|^2 + V(x-\gamma(t)) |\psi(x, t)|^2 \dd x + \int_{\R^3} G(|\psi(x, t)|) \dd x,
$$
where $G(y)=\int F(y) \dd y$.

Setting $V$ in motion does not alter the $L^2$ norm conservation, but energy is no longer conserved (i.e.\ time-independent), only bounded, in this time-dependent case, as we shall see below.

An even more general type of time dependence, which includes the rescaling of the potential (corresponding to possible mass renormalization of
the moving particle), makes sense physically and can also be included in the model.

It then follows that:
\begin{theorem}
Let $\psi$ satisfy equation (\ref{NLS}) in $\R^{3+1}$ with $V(x)$ chosen such that $-\Delta+V$ has neither bound states, nor zero energy resonances.\\
Assume that $V \in L^{3/2, \infty}_0$ 
has a scale $\beta(t)$ and moves along $\gamma(t)=D(t)+2\int_0^t v(s) \dd s$, with $\beta$, $D$, and $v$ in $L^{\infty}$, of locally small variation. 
 $F(y)$ satisfies a growth bound at infinity and vanishes of sufficient order near zero:
$$
F(y)\leq C \max(|y|^{7/3}, |y|^5),\ F'(y) \leq C \max(|y|^{4/3}, |y|^4).
$$
Then $\psi$  satisfies the endpoint Strichartz estimates, for all small enough initial data in $H^1$.
\end{theorem}
The size condition needed to ensure global existence depends on $F$, $V$, and $D$. For large initial data we only obtain local well-posedness, because we are not assuming here that $F$ has any particular sign.

The condition on $F$ ensures that one can bound the contribution of the semilinear term by a combination of the following inequalities:
$$
\|\psi^{7/3}\|_{L^2_t L^{6/5, 2}_x} \les \|\psi\|_{L^{\infty}_t L^2_x} \|\psi\|_{L^2_t L^{6, 2}_x}^{4/3},\ \|\psi^5\|_{L^2_t L^{6/5, 2}_x} \les \|\psi\|_{L^{\infty}_t L^{6, 2}_x} \|\psi\|_{L^2_t L^{6, 2}_x},
$$
and the energy identity, where the energy
$$
E[\psi](t) = \frac 1 2  \int_{\R^3} |\dl \psi(x, t)|^2 + V(x-\gamma(t)) |\psi(x, t)|^2 \dd x + \int_{\R^3} G(|\psi(x, t)|^2) \dd x
$$
is a uniformly bounded (in time)  quantity.

The proof follows from a direct application of the endpoint
Strichartz estimates to the linear and non-homogenous terms of the
equivalent integral equation, by bootstrapping. The energy bound is proved by combining the energy identity with the Strichartz bound above, in the same way as in the linear case. Again, the smallness of the initial data gives control of the nonlinear contribution to the energy estimate.

Furthermore, by writing an iterated expansion, to some large but finite order, using the Duhamel formula, of the contribution of the semilinear term, and applying Strichartz estimates, the decay condition near zero can be improved to $F(y) \sim |y|^{2+\epsilon}$, where $2$ is the Strauss exponent for this equation.

\section{Notations and basic estimates}
\subsection{Notations}
Note that the Galilean coordinate change
$$
\gamma(t) \mapsto \gamma(t) + y + vt
$$
for fixed $y$, $v \in \set R^3$ corresponds to an isometric transformation of the solution
$$
Z(x, t) \mapsto e^{ivx} Z(x + vt + y, t).
$$
Thus, we are always justified in taking $\gamma(0) = 0$ --- and, more generally, in characterizing $\gamma$ through a seminorm that vanishes for linear functions.

Likewise, an $L^2$-unitary dilation of the form
$$
Z \mapsto e^{\beta(x\dl + 3/2)} Z := e^{3/2\beta} Z(e^{\beta} x)
$$
preserves the equation (\ref{1.1}) if we also rescale $V$ by
$$
V \mapsto e^{\beta(x \dl + 2)} Z := e^{2\beta} Z(e^{\beta} x).
$$
This rescaling is dictated by the equation (has to coincide with that of~$-\Delta$).



Let $V = V_1 V_2$, where $V_1 = |V|^{1/2}$ and $V_2 = |V|^{1/2} \sgn V$.
For a curve $\gamma:[0, \infty) \to \set R^3$, define the operator $T(\gamma)$ on $L^2_{t, x}$ by
\be\lb{def_t}
(T(\gamma) F)(t) = \int_0^t V_2 e^{i(t-s)\Delta} e^{(\gamma(t)-\gamma(s)) \dl} V_1 F(s) \dd s.
\ee
For such an operator $T$, we denote its integral kernel by $T(t, s)$:
$$
(T F)(t) = \int_\R T(t, s) F(s) \dd s.
$$
Let $P_N$ be Paley-Wiener projections on the dyadic frequencies $N$. We denote Lorentz spaces by $L^{p, q}$ and the homogenous Besov spaces by $\dot B^s_{p, q}$:
$$
\dot B^s_{p, q} = \{f \mid \big\|N^s \|P_N f\|_{L^p}\big\|_{\ell^q_N} < \infty \}.
$$
Also let $\dot H^s = \dot B^s_{2, 2}$; in the Hilbert space case this is the same as $\dot H^s$, but in general they need not coincide. $\dot H^{1/2}$ is a seminormed space and characterizes functions up to a constant.

Also, let\\
* $C$ be the space of continuous, bounded functions, with the $L^{\infty}$ norm;\\
* $\mc M$ be the set of finite-mass complex-valued Borel measures on $\set R$; $\mc M = C^*$;\\
* $\BV$ be the space of functions of finite variation: $\|f\|_{\BV} = \|f'\|_{\mc M} + \|f\|_{L^{\infty}}$. We take all functions in $\BV$ to be right-continuous at every point.

$\mc M^*$ is the space of bounded Borel measurable functions on $\R$ endowed with the supremum norm; it includes $C$ as a closed subspace. Likewise, an element of $C^*$ is locally a measure and $C^*$ includes $\mc M$ as a closed subspace.

\subsection{Basic estimates}
To begin with, we prove a classical characterization of the Besov spaces $\dot B^{1/2}:= \dot B^{1/2}_{2, 2}$. This approach, which does not involve the Fourier transform, also applies to the mixed norm $\dot H^{1/2}_t X$, for any Banach space $X$.

Consider a smooth function $\chi$ such that $\chi(h) \les_n \langle h \rangle^{-n}$ for each $n$ and $\widehat \chi(\lambda)=0$ for $|\lambda| \leq 1/2$ and $|\lambda| \geq 4$, $\widehat \chi(\lambda)=1$ for $|\lambda| \in [1, 2]$, and $\sum_N \widehat \chi(N \lambda) = 1$. Let $P_{N} f(t)= N^{-1} \chi(N^{-1}t) * f(t)$ be the corresponding Paley-Wiener projections.

\begin{lemma}\lb{besov} Let $X$ be a separable Banach space and
$$
\dot B^s X = \{f \mid N^{s} \|P_N f\|_{L^2_t X} \in~\ell^2_N\}.
$$
Then, for $0<s<1$,
$$
\dot B^s X = \{f \mid \big( 2^{sk} \sup_{|h| \leq 2^{-k}} \|f(t+h)-f(t)\|_{L^2_t X} \big) := g(2^k) \in \ell^2_k\}
$$
with equivalent norms:
$$
\|f\|_{\dot B^s_t X} \sim \|g\|_{\ell^2_k}.
$$
\end{lemma}
\begin{observation} Another way of writing this is, for $0<s<1$,
$$
\|f\|_{\dot B^s X}^2 \sim \int_{\R \times \R} \frac {\|f(t_1)-f(t_2)\|_X^2}{|t_1-t_2|^{2s}} \dd t_1 \dd t_2.
$$
In particular, it follows that $f(t) \mapsto f(1/t)$ is an isometry of $\dot B^{1/2}X$.
\end{observation}
\begin{proof} In one direction, assume that $g(2^k) \in \ell^2_k$. Then
$$\begin{aligned}
\|P_{N^{-1}} f\|_{L^2_t X} &\leq \int_\R N \chi(Nh) \|f(t+h)-f(t)\|_{L^2_t X} \dd h \\
&\les N^{1+s} \int_{|h| \leq N^{-1}} \sup_{h \leq N^{-1}} \chi(Nh) g(N^{-1}) \dd h + \\
&+ \sum_{k=1}^\infty \int_{|h| \in [\frac {2^{k-1}} N, \frac {2^k} N]} 2^{-ks} N^{1+s} \sup_{h \in [\frac {2^{k-1}} N, \frac {2^k} N]} \chi(Nh) g(2^k/N) \dd h \\
&\les N^{s} g(N^{-1}) + \sum_{k=1}^\infty 2^{k(1-s)} N^{s} 2^{-nk} g(2^k/N).
\end{aligned}$$
$N^{-s} P_{N^{-1}} f$ is bounded by a convolution of $g(N^{-1}) \in \ell^2$ with an integrable kernel, for $n > 1-s$. Thus, $\|N^{-s} P_{N^{-1}} f\|_{L^2_t X} \in \ell^2_N$, so $f \in \dot B^{s}X$.

Conversely, begin by noting that if $f \in L^2_t X$, then $\lim_{h \to 0} \|f(t+h)-f(t)\|_{L^2_t X} = 0$. Indeed, since $X$ is separable, $f$ can be approximated by a step function that takes finitely many values, for which this conclusion follows right away.

Then, $\lim_{\epsilon \to 0} \|\epsilon^{-1} \chi(\epsilon^{-1} t) * f - f\|_{L^2_t X} = 0$. By the same reasoning, we obtain that $\lim_{R \to \infty} \|R^{-1} \chi(R^{-1} t) * f\|_{L^2_t X} = 0$. This implies that
$$
\lim_{k, \ell \to \infty} \sum_{N=2^{-k}}^{2^{\ell}} P_N f = f.
$$
Because $\lambda \widehat \chi(\lambda) = \phi(\lambda) \widehat \chi(\lambda)$ for some smooth, compactly supported $\phi$, it follows that $\|\partial_t P_1 f\|_{L^2_t X} \les \|f\|_{L^2_t X}$, so $\|\partial_t P_N f\|_{L^2_t X} \les N \|P_N f\|_{L^2_t X}$.

Let $f_N := N^{s} \|P_N f\|_{L^2_X}$. Then
$$\begin{aligned}
g(2^k)&=2^{ks} \sup_{|h| \leq 2^{-k}} \|f(t+h)-f(t)\|_{L^2_t X} \\
&\les 2^{ks} \sum_N \sup_{|h| \leq 2^{-k}} \|P_N f(t+h) - P_N f(t)\|_{L^2_t X} \\
&\les 2^{ks} \Big(2^{-k} \sum_{N \leq 2^k} \|\partial_t P_N f\|_{L^2_t X} + \sum_{N\geq 2^k} \|P_N f\|_{L^2_t X} \Big) \\
&\les 2^{k(s-1)} \sum_{N \leq 2^{k+1}} N^{1-s} f_N + 2^{ks} \sum_{N \geq 2^k} N^{-s} f_N.
\end{aligned}$$
Since this is the convolution of $(f_N)_N$ with an integrable kernel and $\|(f_N)_N\|_{\ell^2_N} \les \|f\|_{\dot H^s X}$, it follows that $g(2^k) \in \ell^2_k$.
\end{proof}

For a Banach space $X$, we also define $BV_t X$ by
$$
BV_t X = \{f \in L^{\infty}_t X \mid \sup_n \sup_{t_1 < t_2 < \ldots < t_n} \sum_{j=1}^{N-1} \|f(t_{j+1})-f(t_j)\|_X < \infty\}.
$$
We require that functions in $BV_t X$ should be right-continuous. For $f \in BV_t X$, one can define a measure $|f'| \in \mc M$ by
$$
|f'|(I) := \sup_n \sup_{t_1 < t_2 < \ldots < t_n \in I} \sum_{j=1}^{N-1} \|f(t_{j+1})-f(t_j)\|_X.
$$
Then $f' = f_{\infty} |f'|$, with $f_{\infty} \in L^{\infty}(d|f'|, X)$. Conversely, any function with such a decomposition is in $BV_t X$.

Although false in general, for a Hilbert space it is true that $\dot B^{s} = \dot H^{s}$:
\begin{lemma}
Let $X$ be a Hilbert space and $f \in L^2 X$. Then $\|f\|_{L^2_t X} = \|\widehat f\|_{L^2_t X}$, where $\widehat f(\lambda) = \int_{-\infty}^{\infty} e^{-it\lambda} f(t) \dd t$. Moreover, $\dot B^s X = \dot H^s X$, where
$$
\dot B^s X = \{f \mid N^s\|P_N f\|_{L^2 X} \in \ell^2_N\},\ \dot H^s X = \{f \mid |\lambda|^s\widehat f \in L^2_{\lambda} X\}.
$$
\end{lemma}
\begin{proof} Since $\langle f, g \rangle = \langle \widehat f, \widehat g \rangle$, setting $f=g$ we obtain Plancherel's identity (only possible when $X$ is a Hilbert space).

For the second part, use Plancherel's identity for $P_N f$, for each $N$.
\end{proof}

We allow the potential to move along a trajectory of class $\dot H^{1/2} \cap C$ or, more generally, $\CC = (\dot H^{1/2} \cap C) + \BV$.


Functions in $\CC$ have left limits at each point, are right-continuous, possess half a derivative, and are bounded: $\CC \subset L^{\infty}$.

The space of trajectories $\CC$ must have the following properties:
$$
\|\chi_{[0, \infty)}(t) f(t)\|_\CC \les \|f\|_\CC;\ \Big\|\int_{-\infty}^t f'(s) g(s) \dd s\Big\|_{\CC_t} \les \|f\|_\CC \|g\|_\CC.
$$
Ours is a minimal choice in this regard. The former property makes it possible to have $L^2$  initial data and $L^1_t L^2_x$ source terms in the equation; the latter caps the interactions, due to $\gamma$, between bound and dispersive states.

The derivatives of functions in $\CC$ are in $\CC'= \dot H^{-1/2} \cap \partial_t C + \mc M$, where $\partial_t C$ is a space of distributions. The dual of $\CC$ is $\CC^*= (\dot H^{-1/2} + C^*) \cap \partial_t \mc M^*$; thus, $\CC' \subset \CC^*$.

For any closed interval or half-line $I \subset \R$ let $\dot H^{1/2}(I) \cap C(I)$ be the space of restrictions to $I$ of functions in $\dot H^{1/2} \cap C$. It is relatively straightforward to prove that several alternative definitions give rise to the same space.

We denote the space defined in the same manner as $\CC$ on any interval $I$ by $\CC(I) := (\dot H^{1/2}(I) \cap C(I)) + \BV(I)$. Subsequent statements are meant to be equally valid in this setting, with a constant independent of the size of the interval $I$.

Finally, in addition to the scalar space $\CC$, we also consider functions of the form $(\dot H^{1/2}_t X_1 \cap C_t X_2) + \BV_t X_2$, where $X_2 \subset X_1$ are Hilbert spaces.


The properties of $\CC$ are collected in the following lemma:
\begin{lemma}\lb{lemma24} The following statements are true in $\Gamma$ and in $\Gamma(I)$:
\begin{align*}
\|\chi_{[t_0, \infty)}(t) f(t)\|_\CC &\les \|f\|_\CC &\|\chi_{[t_0, \infty)}(t) f(t)\|_{\CC'} &\les \|f\|_{\CC'}, \\
\|f g\|_\CC &\les \|f\|_\CC \|g\|_\CC &\|f g\|_{\CC'} &\les \|f\|_\CC \|g\|_{\CC'}, \\
\Big\|\int_{-\infty}^t f(s) \dd s \Big\|_{\CC_t} &\les \|f\|_{\CC'}, & \\
\|e^{if(t)}-1\|_{\dot H^{1/2} \cap C} &\les \|f\|_{\dot H^{1/2} \cap C} & \|e^{if(t)}\|_{\CC} &\les 1+ \|f\|_{\CC}^2, \\
\|e^{if(t) M} - I\|_{\dot H^{1/2} \cap C} &\les_M \|f\|_{\dot H^{1/2} \cap C} & \|e^{if(t) M}\|_{\CC} &\les_M 1 + \|f\|_{\CC}^2, \\
\|e^{\gamma(t)\dl}f\|_{\dot H^{1/2}_t X} &\les \|\gamma\|_{\dot H^{1/2}} \|\dl f\|_X & \|e^{\gamma(t)\dl}f\|_{\CC_t X} &\les (1+\|\gamma\|_{\CC_t}^2) \|\dl^2 f\|_X;
\end{align*}
furthermore,
$$
\|\chi_{[0, \infty)}(t) f(t)\|_{(\dot H^{1/2}_t X_1 \cap C_t X_2) + \BV_t X_2} \les \|f\|_{(\dot H^{1/2}_t X_1 \cap C_t X_2) + \BV_t X_2},
$$
where $M$ is a selfadjoint matrix, $X$ is a translation-invariant Banach space and $X_2 \subset X_1$ are Hilbert spaces. More generally,
$$
\|e^{D(t) \dl} F(t)\|_{X} \les \|D\|_{\dot H^{1/2}} \|\dl F\|_{L^{\infty}_t X} + \|F\|_{\dot H^{1/2}_t X}.
$$
and if $X$ is a Banach lattice
$$
\|e^{iv(t) x} F(t)\|_{\dot H^{1/2} X} \les \|v\|_{\dot H^{1/2}} \|x F\|_{L^{\infty}_t X} + \|F\|_{\dot H^{1/2}_t X}.
$$
Finally, let $e^{\beta (x\dl + 3/2)}$ be the unitary dilation operator
$$
e^{\beta (x\dl + 3/2)} f(x) = e^{3/2 \beta} f(e^{\beta} x).
$$
Then
$$
\|e^{\beta (x\dl + 3/2)} F\|_{\dot H^{1/2}_t L^2_x} \les \|\beta\|_{\dot H^{1/2}} (\|x \dl F\|_{L^{\infty}_t L^2_x} + \|F\|_{L^{\infty}_t L^2_x}) + \|F\|_{\dot H^{1/2}_t L^2_x}.
$$
\end{lemma}
\begin{observation}
More generally, let $S(t)$ be a $\dot H^{1/2}$ family of isometries on $X$, of infinitesimal generator $s$, in the sense that $S(t)=\exp(\pi(t)s)$ is defined by $\partial_t S(t) = \dot \pi(t) s S(t)$ and $\pi(t) \in \dot H^{1/2}$. Then
$$
\|S(t) F(t)\|_{\dot H^{1/2} X} \les \|\pi\|_{\dot H^{1/2}} \|s F\|_{L^{\infty}_t X} + \|F\|_{\dot H^{1/2} X}.
$$
However, we shall not need this degree of generality in the sequel.
\end{observation}
\begin{proof}
The classes $\CC$ and (by duality) $\CC^*$ are preserved by multiplication with the cutoff functions $\chi_{[t_0, \infty)}$, $t_0 \in \R$: this is the case for $\BV$ and also
$$
\|\chi_{[t_0, \infty)} f\|_{(\dot H^{1/2} \cap C) + \BV} \les \|f\|_{\dot H^{1/2} \cap C}.
$$
We prove this statement as follows: for simplicity, take $t_0=0$.
Let $\chi_1$ be a smooth function such that $\chi_1(t)=1$ for $t \geq 2$ and $\chi_1(t) = 0$ for $t \leq 1$ and let $\chi_2(t) = 1-\chi_1(t)-\chi_1(-t)$. For $f \in \dot H^{1/2} \cap C$, write 
$$\begin{aligned}
\chi_{[0, \infty)}(t) f(t)&=f_1(t)+f_2(t),\\
f_1(t) &= \chi_1(t) f(t) + \chi_{[0, \infty)}(t) \chi_2(t) (f(t) -  f(0)),\\
f_2(t)&=\chi_{[0, \infty)}(t) \chi_2(t) f(0),
\end{aligned}$$
where $f_1$ is continuous and $f_2$ is a step function. Clearly $f_2 \in \BV$ and
$$
\|\chi_1(t) f(t)\|_{H^{1/2}_t} + \|\chi_2(t) f(t)\|_{H^{1/2}_t} \les \|f\|_{H^{1/2}}, \|\chi_2(t) f(0)\|_{H^{1/2}_t} \les \|f\|_C.
$$
By Lemma \ref{besov}, since
$$
\|\chi_{[0, \infty)}(t) g(t)\|_{\dot H^{1/2}} \les \big\| 2^{k/2} \sup_{|h| \leq 2^{-k}} \|\chi_{[0, \infty)}(t+h) g(t+h)-\chi_{[0, \infty)}(t) g(t)\|_{L^2_t X} \big\|_{\ell^2_k}
$$
and
$$
\|\chi_{[0, \infty)}(t+h) g(t+h)-\chi_{[0, \infty)}(t) g(t)\|_{L^2_t} \les \|g(t+h)-g(t)\|_{L^2_t} + \|\chi_{[0, |h|]}(t) g(t)\|_{L^2_t},
$$
it follows that
$$
\|\chi_{[0, \infty)}(t) g(t)\|_{\dot H^{1/2}} \les \|g\|_{\dot H^{1/2}} + \Big(\int_{\R} \frac{|g(t)|^2}{|t|} \dd t\Big)^{1/2}.
$$
The next step is proving that
\be\lb{h12}
\Big(\int_{\R} \frac{|\chi_2(t) (f(t)-f(0))|^2}{|t|} \dd t\Big)^{1/2} \les \|\chi_2(t) (f(t)-f(0))\|_{H^{1/2}}.
\ee
For $g\in H^1_0((-\infty, 0) \cup (0, \infty))$, interpolating between $\int_{\R} |g(t)|^2 \dd t = \|g\|_2^2$~and%
$$
\int_{\R} \frac{|g(t)|^2}{|t|^2} \dd t \les \|g\|_{H^1}^2,
$$
(see \cite{taobook}, the Agmon division lemma) we obtain
$$
\int_{\R} \frac{|g(t)|^2}{|t|} \dd t \les \|g\|_{H^{1/2}}^2.
$$
Let $g \in \mc S$ have $g(0)=0$; then
$$
\int_{\R} \frac{|g(t)|^2}{|t|} \dd t \les \|g\|_{H^{1/2}}^2.
$$
In particular, this is the case for $\chi_2(t) (f(t)-f(0))$, for any $f \in C^{\infty}$. Then
$$
\int_{\R} \frac{|\chi_2(t) (f(t)-f(0))|^2}{|t|} \dd t \les \|\chi_2(t) (f(t)-f(0))\|_{H^{1/2}}^2.
$$
This proves (\ref{h12}); since $\|\chi_2(t) f(t)\|_{H^{1/2}} \les \|f\|_{H^{1/2}}$ and $\|\chi_2(t) f(0)\|_{H^{1/2}} \les \|f\|_C$, we obtain
$$
\|\chi_{[0, \infty)} f\|_{(\dot H^{1/2} \cap C) + \BV} \les \|f\|_{H^{1/2} \cap C}.
$$
Rescaling $f$ so that the $L^2$ norm goes to zero and the $\dot H^{1/2} \cap C$ norm is constant, we can replace $H^{1/2}$ by $\dot H^{1/2}$. Approximating any $f\in \dot H^{1/2} \cap C$ by $C^{\infty}$ functions, the conclusion follows.

The same considerations apply in the more general setting of Hilbert spaces $X_1$ and $X_2$, where the Plancherel identity is valid.

This proves that $\CC$ is an algebra. Indeed, both $\dot H^{1/2} \cap C$ and $\BV$ are algebras, when considered separately. Concerning a product between $f \in \dot H^{1/2} \cap C$ and $g \in \BV$, we proceed as follows: write
$$
(f g)(t) = g(-\infty) f(t) + \int_\R (\chi_{[\tau, \infty)} f)(t) \dd g'(\tau).
$$
Then, by Minkowski's inequality,
$$\begin{aligned}\lb{2.6}
\|fg\|_{\CC} &\leq \|g\|_{L^{\infty}_t} \|f\|_{\dot H^{1/2} \cap C} + \|g'\|_{\mc M_t} \sup_t \|\chi_{[t, \infty)} f\|_{\CC} \\
&\les (\|g'\|_{\mc M_t} + \|g\|_{L^{\infty}}) \|f\|_{\dot H^{1/2} \cap C}.
\end{aligned}$$
Hence $\CC$ is an algebra, which also implies, by duality, that $\|fg\|_{\CC^*} \les \|f\|_{\CC} \|g\|_{\CC^*}$.

Assume that $g \in \CC'$; then $g = g_1 + g_2$ with $g_1 \in \dot H^{-1/2} \cap \partial_t C_t$ and $g_2 \in \mc M$. Then $\chi_{[0, \infty)}(t) g_2(t) \in \mc M$ and $\chi_{[0, \infty)}(t) g_1(t) \in \CC^*$. If $G$ is a continuous antiderivative of $g_1$, then
$$
\tilde G(t) = \chi_{[0, \infty)}(t) (G(t)-G(0))
$$
is a continuous antiderivative for $\chi_{[0, \infty)} g_1$ --- which shows that $\chi_{[0, \infty)} g_1 \in \CC'$.

More generally, take $f \in \CC$ and $g \in \CC'$, $g=g_1 + g_2$, $g_1 \in \dot H^{-1/2} \cap \partial_t C_t$ and $g_2 \in \mc M$. Then $\|f g_2\|_{\mc M} \les \|f\|_{C + \BV} \|g_2\|_{\mc M}$. Let $f=f_1+f_2$, where $f_1 \in \dot H^{1/2} \cap C$ and $f_2 \in \BV$. We write
$$
f_2 g_1 = f_2(-\infty) g_1(t) + \int_\R (\chi_{[\tau, \infty)} g_1)(t) \dd f_2'(\tau),
$$
implying that $f_2 g_1 \in \Gamma$. Finally, approximating $f_1$ and $g_1$ by Schwartz functions, we obtain that $f_1 g_1$ is in the $\Gamma^*$-closure of the Schwartz space, hence in $(\dot H^{-1/2} \cap \partial_t C) +  L^1 \subset \Gamma'$.

Next, we show that the antiderivative of a function in $\CC'$ belongs to $\CC$. Decompose $f \in \CC'$ into $f = f_1 + f_2$, where $f_1 \in \dot H^{-1/2} \cap \partial_t C$ and $f_2 \in \mc M$. Then the antiderivative of $f_2$ is in $\BV \subset \CC$ by definition. If $f_1 \in \dot H^{-1/2}$,
$$
\int_t^{\infty} f_1(s) \dd s - \int_{-\infty}^t f_1(s) \dd s \in \dot H^{1/2}_t.
$$
The Fourier transform of this convolution kernel is $1/\xi$. Then
$$
\int_{-\infty}^t f_1(s) \dd s = \frac 1 2 \bigg(\int_{-\infty}^{\infty} f_1(s) \dd s - \Big(\int_t^{\infty} f_1(s) \dd s - \int_{-\infty}^t f_1(s) \dd s\Big)\bigg)
$$
is in $\dot H^{1/2}$, up to the constant term $\int_{-\infty}^{\infty} f_1(s) \dd s$, since $f_1 \in \partial_t C_t$. We put this constant in $\BV$. By definition, $f_1$'s antiderivative is in $C$ as well, hence in $\dot H^{1/2} \cap C \subset \CC$.

Since $\CC$ is an algebra, $\|e^{if}\|_\CC \leq e^{\|f\|_\CC}$; however, a better estimate is available when $f$ is real-valued: following Lemma \ref{besov}, since $|e^{ia} - e^{ib}| \leq |a-b|$,
$$
\|e^{if} - 1\|_{\dot H^{1/2} \cap C} \les \|f\|_{\dot H^{1/2} \cap C}.
$$
More generally, let $f\in \Gamma = f_1 + f_2$, $f_1 \in \dot H^{1/2} \cap C$, $f_2 \in \BV$. Then
$$\begin{aligned}
&e^{i(f_1(t_2) + f_2(t_2))} - e^{i(f_1(t_1) + f_2(t_1))} - \int_{t_1}^{t_2} i f_2'(t) e^{i(f_1(t) + f_2(t))} \dd t = \\
&= \int_{t_1}^{t_2} i f_2'(t) e^{if_2(t)} (e^{if_1(t_1))}-e^{if_1(t)}) \dd t + e^{if_2(t_2)} (e^{if_1(t_2)}-e^{if_1(t_1)}) \\
&\les (1+\|f_2\|_{\BV_t}) \sup_{t \in [t_1, t_2]} |f_1(t) - f_1(t_1)|.
\end{aligned}$$
By Lemma \ref{besov}, it follows that
$$
\Big\|e^{i(f_1(t) + f_2(t))} - \int_0^t i f_2'(\tau) e^{i(f_1(\tau) + f_2(\tau))} \dd \tau\Big\|_{\dot H^{1/2}_t} \les (1+\|f_2\|_{\BV}) \|f_1\|_{\dot H^{1/2}},
$$
so $\|e^{if(t)}\|_{\CC_t} \les 1+\|f\|_{\CC_t}^2$.

When $M$ is a selfadjoint matrix, we first bring it to a diagonal form, $M = CDC^{-1}$; then $e^{if(t)M} = C e^{if(t) D} C^{-1}$ and the conclusion holds for each individual eigenvalue.

Likewise, note that
$$
\|e^{D(t+h) \dl} f - e^{D(t) \dl} f\|_{X} \leq |D(t+h) - D(t)| \|\dl f\|_X;
$$
then the characterization of $\dot H^{1/2}$ given in Lemma \ref{besov} establishes that
$$
\|e^{\gamma(t) \dl} f\|_{\dot H^{1/2} X} \les \|\gamma\|_{\dot H^{1/2}_t} \|\dl f\|_X.
$$
More generally, for a time-dependent function $F(t)$,
$$
\|e^{D(t+h) \dl} F(t+h) - e^{D(t) \dl} F(t)\|_{X} \les |D(t+h) - D(t)| \|\dl F\|_{L^{\infty}_t X} + \|F(t+h)-F(t)\|_X,
$$
which shows that
$$
\|e^{D(t) \dl} F(t)\|_{X} \les \|D\|_{\dot H^{1/2}} \|\dl F\|_{L^{\infty}_t X} + \|F\|_{\dot H^{1/2}_t X}.
$$
Assuming that $X$ is a Banach lattice, we also obtain
$$
\|e^{v(t+h) x} F(t+h) - e^{v(t) x} F(t)\|_{X} \leq |v(t+h) - v(t)| \|x F\|_{L^{\infty}_t X}  + \|F(t+h)-F(t)\|_X,
$$
implying that
$$
\|e^{v(t) x} F(t)\|_{X} \les \|v\|_{\dot H^{1/2}} \|x F\|_{L^{\infty}_t X} + \|F\|_{\dot H^{1/2}_t X}.
$$
The statement about dilation is proved in the same manner.

Finally, consider $F(t) = e^{\gamma(t)\dl}f$, $\gamma \in \CC = \gamma_1 + \gamma_2$, $\gamma_1 \in \dot H^{1/2} \cap C$, $\gamma_2 \in \BV$. Since
$$\begin{aligned}
&\Big\|e^{\gamma(t_2)\dl}f - e^{\gamma(t_1)\dl} f - \int_{t_1}^{t_2} i \gamma_2'(t)\dl e^{\gamma(t)\dl} f \dd t\Big\|_X \les \\
&\les (1+\|\gamma_2\|_{\BV}) \sup_{t \in [t_1, t_2]} |\gamma_1(t) - \gamma_1(t_1)| \|\dl^2 f\|_X,
\end{aligned}$$
it follows that $\|e^{i\gamma(t)\dl}f\|_{\CC_t X} \les (1+\|\gamma\|_{\CC_t}^2) \|\dl^2 f\|_X$.

To retrieve the same results on any interval $I$, it suffices to note that $\CC(I)$ is the restriction to $I$ of $\CC$.
\end{proof}

\section{Proof of the main statements}\lb{sect_3}
Let $V = V_1 V_2$, where $V_1 = |V|^{1/2}$ and $V_2 = |V|^{1/2} \sgn V$.
For a curve $\gamma:[0, \infty) \to \R^3$, a variable scale $\beta:[0, \infty) \to \R$, a velocity $v:[0, \infty) \to \R^3$, and a phase $\alpha:[0, \infty) \to \R$, let $\pi$ be the parameter path $\pi = (\gamma, v, \beta, \alpha)$ and define the isometry
$$
S_{\pi}(t) = e^{i\alpha(t)} e^{\beta(t)(x\dl + 3/2)} e^{iv(t) x} e^{\gamma(t) \dl}
$$
and the operator $T(\pi)$ on $L^2_{t, x}$ by
\be
(T(\pi) F)(t) = \int_0^t V_2 S_{\pi}(t) e^{i(t-s)\Delta} S_{\pi}(s)^{-1} V_1 F(s) \dd s.
\ee
Here $e^{\beta (x\dl + 3/2)}$ is the $L^2$-unitary dilation operator
$$
e^{\beta (x\dl + 3/2)} f(x) := e^{3\beta/2} f(e^{\beta} x);
$$
$e^{\gamma \dl}$ represents the translation $e^{\gamma \dl} f(x) = f(x+\gamma)$.

First, we show that $T(\pi)$ is a continuous mapping 
to $\B(L^2_{t, x}, L^2_{t, x})$ in the norm topology. The proper norm to consider for $\pi$ is $\|\pi\|_{\Pi} = \|\alpha\|_{L^{\infty}+\dot W^{1, \infty}} + \|\gamma\|_{L^{\infty}+\dot W^{1, \infty}} + \|v\|_{L^{\infty}} + \|\beta\|_{L^{\infty}}$.


\begin{lemma}\lb{lemma2.1} Assume $V \in \langle x \rangle^{-4} L^1 \cap L^{\infty}$. Then
$$
\|T(\pi) - T(\pi_0)\|_{\B(L^2_{t, x}, L^2_{t, x})} \les \|V\|_{\langle x \rangle^{-4} L^1 \cap L^{\infty}} \|\pi-\pi_0\|_{\Pi}^{2/5}.
$$
More generally, if $V \in L^{3/2, \infty}_0$ (the weak-$L^{3/2}$ closure of the set of bounded, compactly supported functions), then for fixed $\gamma_0$
$$
\lim_{\|\pi - \pi_0\|_{\Pi}\to 0} \|T(\pi) - T(\pi_0)\|_{\B(L^2_{t, x}, L^2_{t, x})} = 0.
$$
\end{lemma}
In the latter case, the rate of convergence depends on the profile of $V$ (there is no exact estimate).
\begin{proof}
To begin with, assume $V$ is in $\langle x \rangle^{-4} L^1 \cap L^{\infty}$. Note that $\|T(\pi_0)(t, s)\|_{\B(L^2_x, L^2_x)}$ and $\|T(\pi)(t, s)\|_{\B(L^2_x, L^2_x)}$ are uniformly bounded by $\langle t-s \rangle^{-3/2}$:
$$
\|T(\pi_0)(t, s)\|_{\B(L^2_x, L^2_x)} + \|T(\pi)(t, s)\|_{\B(L^2_x, L^2_x)} \les \|V\|_{L^1 \cap L^\infty} \langle t-s \rangle^{-3/2}.
$$
We estimate $\|T(\gamma) - T(\gamma_0)\|_{\B(L^2_{t, x}, L^2_{t, x})}$ by
$$
\|T(\gamma) - T(\gamma_0)\|_{\B(L^2_{t, x}, L^2_{t, x})} \leq \int_0^{\infty} \sup_{t-s = \tau} \|T(\gamma)(t, s) - T(\gamma_0)(t, s)\|_{\B(L^2_x, L^2_x)} \dd \tau.
$$
In particular,
$$
\|T(\pi)(t, s)\|_{\B(L^2_{t, x}, L^2_{t, x})} \les \|V\|_{L^1 \cap L^{\infty}},
$$
so with no loss of generality we assume that $\|\pi-\pi_0\|_{\Pi} \leq 1$.

If $\delta_0$ is Dirac's measure at zero, then the fundamental solution, respectively Green's function for the free Schr\"{o}dinger equation are
$$
e^{it \Delta} \delta_0(x) = (4 \pi i t)^{-3/2} e^{i\frac{|x|^2}{4t}},\ e^{it \Delta} (x, y) = (4 \pi i t)^{-3/2} e^{i\frac{|x-y|^2}{4t}}.
$$
This leads to the pointwise bounds, for $|\beta|\leq 1$,
$$\begin{aligned}
|(e^{\beta (x\dl + 3/2)} - 1) e^{it \Delta}(x, y)| &\les t^{-3/2} |\beta| + t^{-5/2} |\beta| |x| |x-y|,\\
|(e^{d \dl} -1) e^{it \Delta} \delta_0(x)| &\les t^{-5/2} (d^2 + d |x|).
\end{aligned}$$
Taking into account the other parameters as well, the difference between the perturbed and the unperturbed kernel is of size
$$\begin{aligned}
&|e^{i\alpha} e^{\beta_1 (x\dl + 3/2)}  e^{-iv_1x} e^{d \dl} e^{it \Delta} e^{iv_2y} e^{\beta_2 (x\dl + 3/2)}(x, y) - e^{it\Delta}(x, y)| \les \\
&\les t^{-5/2} (|\beta_1| |x| |x-y| + |\beta_2| |y| |x-y| + |d|^2 + |d| |x-y|) + \\
&+ t^{-3/2} (|v_1| |x| + |v_2||y|+ |\alpha|).
\end{aligned}$$
For each $t$ and $s$, then,
\begin{equation}\lb{bd1}\begin{aligned}
&\|T(\pi)(t, s) - T(\pi_0)(t, s)\|_{\B(L^2_x, L^2_x)} \les |t-s|^{-5/2} \|V\|_{\langle x \rangle^{-4} L^1} \cdot \\
&\big(|(\gamma(t)-\gamma(s)) - (\gamma_0(t)-\gamma_0(s))|^2 + |(\gamma(t)-\gamma(s)) - (\gamma_0(t)-\gamma_0(s))| \big) + \\
&|t-s|^{-3/2} \|V\|_{\langle x \rangle^{-1} L^1} \big(|(v(t)-v_0(t)|+|v(s)-v_0(s)|\big) + \\
&|t-s|^{-3/2} \|V\|_{\infty} |(e^{i\alpha(t)}-\alpha(s))-(\alpha_0(t)-\alpha_0(s))| \\
&\les (|t-s|^{-5/2} + |t-s|^{-1/2}) \|\pi-\pi_0\|_{\Pi}.
\end{aligned}
\ee
Let $0<\epsilon<R<\infty$ and consider three cases for $t-s \in [0, \infty)$: if $t-s<\epsilon$, we use the bound
$$
\|T(\pi)(t, s) - T(\pi_0)(t, s)\|_{\B(L^2_x, L^2_x)} \les \|V\|_{L^\infty}.
$$
If $t-s>R$, we use the bound
$$
\|T(\pi)(t, s) - T(\pi_0)\|_{\B(L^2_x, L^2_x)} \les \|V\|_{L^1} (t-s)^{-3/2},
$$
while if $t-s \in [\epsilon, R]$ we use (\ref{bd1}). 
This leads to
\begin{align*}
&\|T(\pi) - T(\pi_0)\|_{\B(L^2_{t, x}, L^2_{t, x})} \les \\
&\les \|V\|_{\langle x \rangle^{-4} L^1 \cap L^{\infty}} \big(\epsilon + R^{-1/2} + \|\pi - \pi_0\|_{L^{\infty} + \dot W^{1, \infty}} (\epsilon^{-3/2} + R^{1/2})\big).
\end{align*}
By setting $\epsilon=\|\pi-\pi_0\|_{\Pi}^{2/5}$ and $R=\|\pi-\pi_0\|_{\Pi}^{-1}$, we get
$$
\|T(\pi) - T(\pi_0)\|_{\B(L^2_{t, x}, L^2_{t, x})} \les \|V\|_{\langle x \rangle^{-4} L^1 \cap L^{\infty}} \|\pi-\pi_0\|_{\Pi}^{2/5}.
$$

When $V \in L^{3/2, \infty}_0$, we proceed by approximation. Take $V^n = V_1^n V_2^n$ and two sequences $(V_1^n)_n$ and $(V_2^n)_n$, such that $V_1^n \to V_1$, $V_2^n \to V_2$ in the $L^{3, \infty}$ norm, but $V_1^n$ and $V_2^n$ are bounded of compact support. The result is true for each approximation, so it is still true in the limit.
\end{proof}

In case $-\Delta+V$ has bound states, we replace the functions $V_1$ and $V_2$ in (\ref{def_t}) by the operators $\tilde V_1$ and $\tilde V_2$ given by the following technical lemma:
\begin{lemma}\lb{lemma32}
Consider $V \in L^{3/2, \infty}_0(\set R^3)$ and $H = -\Delta + V$ such that $H$ has no embedded eigenvalues or threshold resonances. Then there exists a decomposition
\be
V - P_p (H-i\delta) = \tilde V_1 \tilde V_2,
\ee
where $P_p$ is the point spectrum projection, such that $\tilde V_1$, $\tilde V_2^* \in \B(L^2, L^{6/5, 2})$, and the Fourier transform of $I-iT_{\tilde V_2, \tilde V_1}$ is invertible in the lower half-plane up to the boundary, where
\be
(T_{\tilde V_2, \tilde V_1}F)(t) = \int_{-\infty}^t \tilde V_2 e^{-i(t-s)\Delta} \tilde V_1 F(s) \dd s.
\ee
Furthermore, $\tilde V_1$ and $\tilde V_2^*$ can be approximated in the $\B(L^2, L^{6/5, 2})$ norm by operators that are bounded from $L^2$ to $\langle x \rangle^{-N} L^2$, for any fixed $N$.
\end{lemma}
For the proof, we refer the reader to \cite{bec}.





Since $\gamma \in \CC$ need not be differentiable, we consider source terms, in Schr\"{o}dinger's equation, that are merely distributions with respect to $t$ --- but which are supported on a specific space of time-dependent functions.

In the sequel, in the course of the paper, we use two fundamental properties of the free Schr\"{o}dinger evolution --- namely, that it satisfies Strichartz estimates and produces local smoothing. We refer the reader to Theorem 1.13 in \cite{mmt} (applied to the free evolution).

Let $\chi$ be a smooth cutoff function supported on $[1/2, 4]$ such that $\chi(r)=1$ on $[1, 2]$ and $P_N f = N\widehat \chi(N |x|) * f(x)$. Define
$$\begin{aligned}
\|F\|_X^2 &= \sup_{\ell} \big\|\chi(2^{-\ell} x) |x|^{-1/2} P_{2^k} F\big\|_{L^2_t \dot H^{1/2}_x}^2 \\
&\sim \sum_{k=-\infty}^{\infty} 2^k \sup_{\ell} \|\chi(2^{-\ell} x) |x|^{-1/2} P_{2^k} F\|_{L^2_{t, x}}^2.
\end{aligned}$$

The smoothing result that we use for the free Schr\"{o}dinger evolution is then, following \cite{mmt},
\begin{lemma}
$$
\|e^{it\Delta} f\|_{X \cap L^2_t L^{6, 2}_x \cap L^{\infty}_t L^2_x} \les \|f\|_{X' + L^2_t L^{6/5, 2}_x + L^1_t L^2_x}.
$$
\end{lemma}

We first prove the following technical inequalities:
\begin{lemma}\lb{lemma_2.4} Let $B_1(t)$ and $B_2(t)$ be time-dependent operators of the form
$$
B_j(t) = e^{\gamma_j(t) \dl} b_j(t)
$$
with $\gamma_j' \in L^{\infty}$, $b_j \in W^{1, \infty}$, $j=1, 2$. Then
\be\begin{aligned}\lb{2.15}
&\bigg\|\Big\langle \int_{-\infty}^t B_1(t) e^{i(t-s)\Delta} B_2(s) F(s)\dd s, G(t) \Big\rangle\bigg\|_{\CC_t} \les_{B_1, B_2} \\
& \|F\|_{L^2_t L^{6/5, 2}_x + L^1_t L^2_x} \|G\|_{C_t \langle \dl \rangle^{-1} \langle x \rangle^{-1-\epsilon} L^2_x \cap \dot H^{1/2}_t L^2_x};
\end{aligned}\ee
\be\begin{aligned}\lb{2.5}
&\Big\|\int_{-\infty}^t \big\langle B_1(t) e^{i(t-s)\Delta} B_2(s) F(s), G(t) \big\rangle \dd s \Big\|_{\CC_t} \les_{B_1, B_2} \\
&\|F\|_{\dot H^{-1/2}_t \langle \dl \rangle^{-1} \langle x \rangle^{-1-\epsilon} L^2_x \cap \partial_t C_t L^2_x + L^1_t L^2_x} \|G\|_{C_t \langle \dl \rangle^{-1} \langle x \rangle^{-1-\epsilon} L^2_x \cap \dot H^{1/2}_t L^2_x}.
\end{aligned}\ee
\end{lemma}
\begin{proof}

(\ref{2.15}) is an adaptation of the fractional Leibniz rule and we prove it as such. To begin with, consider $F(x, t) \in L^2_t L^{6/5, 2}_x$ and let
$$
F_1(x, t) = \int_{-\infty}^t B_1(t) e^{i(t-s)\Delta} B_2(s) F(y, s) \dd s.
$$
Since $\|B_2(t) F(t)\|_{L^2_t L^{6/5, 2}_x} \les_{B_2} \|F\|_{L^2_t L^{6/5, 2}_x}$, without loss of generality, we could set $B_2 \equiv I$.
Following Lemma \ref{besov},
$$\begin{aligned}
\|\langle F_1(t), G(t)\rangle\|_{\dot H^{1/2}_t} &\les \big\|2^{k/2} \sup_{|h| \leq 2^{-k}} \|\langle F_1(t+h), G(t+h) \rangle - \langle F_1(t), G(t)\rangle\|_{L^2_t}\big\|_{\ell^2_k} \\
&\les \big\|2^{k/2} \sup_{|h| \leq 2^{-k}} \|\langle F_1(t+h) - F_1(t), G(t+h)\rangle\|_{L^2_t}\big\|_{\ell^2_k} + \\
& +\big\|2^{k/2} \sup_{|h| \leq 2^{-k}} \|\langle F_1(t), G(t+h) - G(t)\rangle\|_{L^2_t}\big\|_{\ell^2_k} \\
& \les \|F_1\|_{\dot H^{1/2}_t \langle \dl \rangle \langle x \rangle^{1+\epsilon} L^2_x} \|G\|_{C_t \langle \dl \rangle^{-1} \langle x \rangle^{-1-\epsilon} L^2_x} + \|F_1\|_{C_t L^2_x} \|G\|_{\dot H^{1/2}_t L^2_x}.
\end{aligned}$$
The second term is bounded by Strichartz estimates, since
$$\begin{aligned}
\|F_1(t)\|_{C_t L^2_x} \|G(t)\|_{\dot H^{1/2}_t L^2_x} \les_{B_1, B_2} \|F\|_{L^2_t L^{6/5, 2}_x} \|G(t)\|_{\dot H^{1/2}_t L^2_x}.
\end{aligned}$$
For the first term, observe that
$$
\|\langle x\rangle^{-1/2-\epsilon} F_1\|_{L^2_t \dot H^{1/2}_x} \les_{B_1, B_2} \|F\|_{L^2_t L^{6/5, 2}_x}
$$
and
$$\begin{aligned}
\partial_t F_1(t) &= B_1'(t) \int_{-\infty}^t e^{i(t-s)\Delta} B_2(s) F(y, s) \dd s + \\
&+ B_1(t) B_2(t) F(y, t) + iB_1(t) \Delta \int_{-\infty}^t e^{i(t-s)\Delta} B_2(s) F(y, s) \dd s,
\end{aligned}$$
which implies that
$$
\|\langle x\rangle^{-3/2-\epsilon} \partial_t \langle\dl\rangle^{-3/2} F_1\|_{L^2_{t, x}} \les_{B_1, B_2} \|F\|_{L^2_t L^{6/5, 2}_x}.
$$
By complex interpolation of exponent $1/2$,
we obtain
$$
\|F_1\|_{\dot H^{1/2}_t  \langle \dl \rangle^{1/2} \langle x \rangle^{1+\epsilon} L^2_x} \les_{B_1, B_2} \|F\|_{L^2_t L^{6/5, 2}_x}.
$$
Continuity follows in (\ref{2.15}) since $\langle \dl \rangle^{-1} \langle x \rangle^{-2} L^2 \subset L^2$ and
$$
\|\langle F_1(t), G(t)\rangle\|_{C_t} \les \|F_1\|_{C_t L^2_x} \|G\|_{C_t \langle \dl \rangle^{-1} \langle x \rangle^{-1-\epsilon} L^2_x}.
$$
Next, let $F_1(t) = B(t) e^{it\Delta} f$, $f \in L^2$. Then
$$
\|F_1(t)\|_{C_t L^2_x \cap L^2_t L^{6, 2}_x} \les_U \|f\|_2
$$
and
$$
\partial_t F_1 = B'(t) e^{it\Delta} f + iB(t)\Delta e^{it\Delta} f,
$$
leading to
$$
\|\langle\Delta\rangle^{-1} F_1\|_{\partial_t^{-1} L^2_t \langle x \rangle L^{6, 2}_x} \les \|u\|_{W^{1, \infty}} \|f\|_2.
$$
By the same process as above, we obtain
$$
\|F_1\|_{\dot H^{1/2}_t  \langle \dl \rangle \langle x \rangle^{1+\epsilon} L^2_x \cap C_t L^2_x} \les \|u\|_{W^{1, \infty}} \|f\|_{2}.
$$
Cutting off $F_1$ by $\chi_{[0, \infty)}(t)$, it follows as in Lemma \ref{lemma24} that
$$
\|\chi_{[0, \infty)}(t) F_1(t)\|_{\dot H^{1/2}_t \langle \dl \rangle \langle x \rangle^{1+\epsilon} L^2_x \cap C_t L^2_x + \BV_t L^2_x} \les \|u\|_{W^{1, \infty}} \|f\|_{2}
$$
and, more generally, if $F_1(t) = B_1(t) \int_{-\infty}^t e^{i(t-s)\Delta} B_2(s) F(s) \dd s$,
\be\lb{L2}
\|F_1\|_{\dot H^{1/2}_t  \langle \dl \rangle \langle x \rangle^{1+\epsilon} L^2_x \cap C_t L^2_x + \BV_t L^2_x} \les \|u_1\|_{\U} \|u_2\|_{\U} \|F\|_{L^1_t L^2_x}.
\ee
This completes the proof of (\ref{2.15}), as it implies that
$$
\langle F_1, G \rangle_{\Gamma} \les_{B_1, B_2} \|F\|_{L^1_t L^2_x} \|G\|_{C_t \langle \dl \rangle^{-1} \langle x \rangle^{-1-\epsilon} L^2_x \cap \dot H^{1/2}_t L^2_x}.
$$
The dual of (\ref{L2}) implies that (and is strictly stronger than)
\be\lb{L^2_dual}
\|F_1\|_{C_t L^2_x} \les_{B_1, B_2} \|F\|_{\dot H^{-1/2}_t  \langle \dl \rangle^{-1} \langle x \rangle^{-1-\epsilon} L^2_x \cap \partial_t C_t L^2_x + L^1_t L^2_x}.
\ee
Next, we prove (\ref{2.5}), which gains a full derivative in $t$.
Starting from
$$\begin{aligned}
&\|F_1\|_{\partial_t^{-1} L^2_t \langle\dl\rangle \langle x \rangle L^{6, 2}_x} \les_{B_1, B_2} \|F\|_{L^2_t \langle \dl \rangle^{-1} \langle x \rangle^{-1} L^{6/5, 2}_x}
\end{aligned}$$
and its dual
$$\begin{aligned}
&\|F_1\|_{L^2_t \langle\dl\rangle \langle x \rangle L^{6, 2}_x} \les_{B_1, B_2} \|F\|_{\partial_t L^2_t \langle \dl \rangle^{-1} \langle x \rangle^{-1} L^{6/5, 2}_x},
\end{aligned}$$
by complex interpolation of exponent $1/2$ we obtain, as above,
$$\begin{aligned}
&\|F_1\|_{\dot H^{1/2}_t \langle \dl \rangle \langle x \rangle^{1+\epsilon} L^2_x} \les_{B_1, B_2} \|F\|_{\dot H^{-1/2}_t \langle \dl \rangle^{-1} \langle x \rangle^{-1-\epsilon} L^2_x}.
\end{aligned}$$
We then use
$$
\|\langle F_1, G \rangle\|_{\dot H^{1/2}_t} \les \|F_1\|_{\dot H^{1/2}_t \langle \dl \rangle \langle x \rangle^{1+\epsilon} L^2_x} \|G\|_{C_t \langle \dl \rangle^{-1} \langle x \rangle^{-1-\epsilon} L^2_x} + \|F_1\|_{C_t L^2_x} \|G(t)\|_{\dot H^{1/2}_t L^2_x}.
$$
This establishes the $\dot H^{1/2}$ conclusion of (\ref{2.5}). In order to bound $\|F_1\|_{C_t L^2_x}$, we invoke (\ref{L^2_dual}).
\end{proof}

The subsequent lemma controls the singular terms that appear in the linearized Schr\"{o}dinger equation, which are distributions with respect to time.

\begin{lemma}\lb{lema35}  Let
$$
S(t) = e^{D(t)\dl} e^{iv(t)x} e^{\beta(t)(x\dl + 3/2)} e^{i\alpha(t)},\ \beta, v, D, \alpha \in \dot H^{1/2} \cap C,
$$
and assume that $B_1$, $B_2$ are as in Lemma \ref{lemma_2.4}. Then
\be\begin{aligned}\lb{eqn2.11}
&\Big\| \Big\langle \int_{-\infty}^t S(t)^{-1} B_1(t) e^{i(t-s)\Delta} B_2(s) S(s) F(s)\dd s, g \Big \rangle \Big\|_{\CC_t} \les_{B_1, B_2, \beta, v, D, \alpha} \\
&\|F\|_{L^2_t L^{6/5, 2}_x + L^1_t L^2_x} \|g\|_{\langle \dl \rangle^{-1} \langle x \rangle^{-1-\epsilon} L^2_x};
\end{aligned}\ee
\be\begin{aligned}\lb{eq2.45}
&\Big\|\Big\langle \int_{-\infty}^t S(t)^{-1} B_1(t) e^{i(t-s)\Delta} B_2(s) S(s) \phi(s) f\dd s, g \Big\rangle\Big\|_{\CC_t} \les_{B_1, B_2, \beta, v, D, \alpha} \\
&\|\phi\|_{\CC'} \|f\|_{\langle \dl \rangle^{-1} \langle x \rangle^{-1-\epsilon} L^2_x} \|g\|_{\langle \dl \rangle^{-1} \langle x \rangle^{-1-\epsilon} L^2_x}.
\end{aligned}\ee
\end{lemma}
\begin{proof}

Note that interchanging $e^{D(t)\dl}$ and $e^{iv(t)x}$ in the definition of $S$ only generates a factor of $e^{iD(t) \cdot v(t)}$, which can be absorbed into $e^{i\alpha(t)}$. Likewise, commuting $e^{\beta(t)(x\dl + 3/2)}$ with the other operators only changes $v$ to $\beta v$, $D$ to $\beta D$, leaving the nature of the expression unchanged --- so the order in which we write these operators is unimportant. Furthermore, applying $S(t)$ to $F(t)$ leaves its $L^2_t L^{6, 2}_x$ norm unchanged.

In (\ref{2.15}), let $G(t) = S(t) g = e^{D(t)\dl} e^{iv(t)x} e^{\beta(t)(x\dl + 3/2)} e^{i\alpha(t)} g$. By Lemma~\ref{besov},
\be\lb{eq_g}
\|e^{D(t)\dl} e^{iv(t)x} e^{\beta(t)(x\dl + 3/2)} e^{i\alpha(t)} g\|_{\dot H^{1/2}_t L^2_x} \les_{\beta, v, D, \alpha} \|g\|_{\langle \dl \rangle^{-1} \langle x \rangle^{-1} L^2_x}.
\ee
Furthermore,
\be
\|e^{D(t)\dl} e^{iv(t)x} e^{\beta(t)(x\dl + 3/2)} e^{i\alpha(t)} g\|_{C_t \langle \dl \rangle^{-1} \langle x \rangle^{-1-\epsilon} L^2_x} \les_{\beta, v, D, \alpha} \|g\|_{\langle \dl \rangle^{-1} \langle x \rangle^{-1-\epsilon} L^2_x}.
\ee
This implies the first inequality, (\ref{eqn2.11}).

Concerning (\ref{eq2.45}), letting $F(s) = S(s) \phi(s) f$ in (\ref{2.5}), we show that $F \in \dot H^{-1/2}_s \langle \dl \rangle^{-1} \langle x \rangle^{-1-\epsilon} L^2_x \cap \partial_s C_s L^2_x + L^1_s L^2_x$ by testing it against elements of the dual space: from
$$\begin{aligned}
\int_\R \langle F(s), G(s) \rangle \dd s &= \int_\R \phi(s) \langle S(s) f, G(s) \rangle \dd s \\
&\les \|\phi\|_{\dot H^{-1/2}} \|\langle S(s) f, G(s) \rangle\|_{\dot H^{1/2}_s} \\
&\les \|\phi\|_{\dot H^{-1/2}} \|S(s) f\|_{\dot H^{1/2}_s L^2_x \cap C_s \langle \dl \rangle^{-1} \langle x \rangle^{-1-\epsilon} L^2_x} \|G\|_{\dot H^{1/2}_s \langle \dl \rangle \langle x \rangle^{1+\epsilon} L^2_x \cap L^{\infty}_s L^2_x}
\end{aligned}$$
we obtain, by approximating with smooth functions, since $L^1$ is a closed subspace of $(L^{\infty})^*$, that
$$
\|F\|_{\dot H^{-1/2}_s \langle \dl \rangle^{-1} \langle x \rangle^{-1-\epsilon} L^2_x + L^1_t L^2_x} \les \|\phi\|_{\dot H^{-1/2}} \|S(s) f\|_{\dot H^{1/2}_s L^2_x \cap C_s \langle \dl \rangle^{-1} \langle x \rangle^{-1-\epsilon} L^2_x}.
$$
Likewise,
$$\begin{aligned}
\int_\R \langle F(s), G(s) \rangle \dd s &= \int_\R \phi(s) \langle S(s) f, G(s) \rangle \dd s \\
&\les \|\phi\|_{\CC'} \|\langle S(s) f, G(s) \rangle\|_{\CC_s} \\
&\les \|\phi\|_{\CC'} \|S(s) f\|_{\dot H^{1/2}_s L^2_x \cap C_s L^2_x} \|G\|_{\BV_s L^2_x}
\end{aligned}$$
implies (by approximating with smooth functions) that
$$
\|F\|_{\partial_s C_s L^2_x} \les \|\phi\|_{\CC'} \|S(s) f\|_{\dot H^{1/2}_s L^2_x \cap C_s \langle \dl \rangle^{-1} \langle x \rangle^{-1-\epsilon} L^2_x}.
$$
By Lemma \ref{besov},
$$
\|S(s) f\|_{\dot H^{1/2}_s L^2_x \cap C_s \langle \dl \rangle^{-1} \langle x \rangle^{-1-\epsilon} L^2_x} \les_{\beta, v, D, \alpha} \|f\|_{\langle \dl \rangle^{-1} \langle x \rangle^{-1-\epsilon} L^2}.
$$
Finally, let $G(t) = S(t) g$ in (\ref{2.5}) and treat it by (\ref{eq_g}). We obtain (\ref{eq2.45}).

\end{proof}

The point spectrum of $H=-\Delta+V$ consists of a finite number $N$ of negative energies $E_k$, $1 \leq k \leq N$, with $L^2$-normalized eigenstates $g_1, \ldots, g_N$. 
Let $P_k = \langle \cdot, g_k \rangle g_k$; then $P_p = \sum_{k=1}^N P_k$ is the projection on the point spectrum.

The negative energies are poles of the Birman-Kato operator $(I + R_0(\lambda) V)^{-1}$, which is analytic on $L^{6, 2}$, so by Fredholm's alternative $g_k$ are in $L^{6, 2}$.

However, $g_k$ and their first two derivatives decay exponentially due to Agmon's bound:
\begin{lemma}\lb{lema36}
Assume that $V \in L^{3/2, \infty}_0$ and $f \in L^{6, 2}$ is an eigenfunction of $-\Delta+V$ corresponding to a negative energy $E<0$: $(-\Delta+V) f = E f$. Then
$$
e^{\sqrt{-E}|x|} f \in L^{p_0},\ e^{\sqrt{-E}|x|} \dl f \in L^{p_0}+L^{p_1},\ e^{\sqrt{-E}|x|} \dl^2 f \in L^{p_0} + L^{p_2},
$$
for any $p_0 \in (3, \infty)$, $p_1 \in (3/2, 3)$, $p_2 \in (1, 3/2)$.

Assume that $V \in L^{3/2}$; then $e^{(\sqrt{-E}-\epsilon)|x|} \dl f \in H^{1}$ for any $\epsilon > 0$.
\end{lemma}
\begin{proof} Let $V = V_1 + V_2$, where $V_1 \in L^1 \cap L^{\infty}$ has compact support and $\|V_2\|_{L^{3/2, \infty}}$ is small. Fix $\epsilon>0$, $q \in [1, \infty]$. Then $(-\Delta-E)^{-1}$, given by the $L^{3, \infty}$ convolution kernel
$$
(-\Delta-E)^{-1}(x, y)= \frac{e^{-\sqrt {-E}|x-y|}}{4\pi|x-y|},
$$
is a bounded operator in the spaces
$$\begin{aligned}
\B(e^{-\sqrt{-E}|x|} L^{3/2-\epsilon, q}, e^{-\sqrt{-E}|x|} L^{\frac 9 {4\epsilon}-\frac 3 2, q}) \cap \B(e^{-\sqrt{-E}|x|} L^{1+\epsilon, q}, e^{-\sqrt{-E}|x|} L^{\frac{3+3\epsilon}{1-2\epsilon}, q}).
\end{aligned}$$
Conversely, $V_2$ and $V \in L^{3/2, \infty}$ belong to
$$
\B(e^{-\sqrt{-E}|x|} L^{\frac 9 {4\epsilon}-\frac 3 2, q}, e^{-\sqrt{-E}|x|} L^{3/2-\epsilon, q})
\cap \B(e^{-\sqrt{-E}|x|} L^{\frac{3+3\epsilon}{1-2\epsilon}, q}, e^{-\sqrt{-E}|x|} L^{1+\epsilon, q}).
$$
Since $e^{\sqrt{-E}|x|} V_1 f \in L^{3/2-\epsilon, q} \cap L^{1+\epsilon, q}$, $f$ is the sum of the infinite series
$$
f = \sum_{k=0}^\infty (-1)^k \big((-\Delta-E)^{-1} V_2\big)^k (-\Delta-E)^{-1} V_1 f.
$$
This shows that $e^{\sqrt{-E}|x|} f \in L^{\frac 9 {4\epsilon}-\frac 3 2, q} \cap L^{\frac{3+3\epsilon}{1-2\epsilon}, q}$, so $e^{\sqrt{-E}|x|} V f \in L^{3/2-\epsilon, q} \cap L^{1+\epsilon, q}$. The other two conclusions follow by taking the derivative in $f = (-\Delta-E)^{-1} V f$.

If $V \in L^{3/2}$, then it follows in the same manner that $e^{\sqrt{-E}|x|} f \in L^{\infty}$. Under the further assumption that $V \in L^2$, we obtain that $e^{\sqrt{-E}|x|}Vf \in L^2$, so $\widehat{Vf}$ has an analytic continuation to $|\Im \xi| < \sqrt{-E}$. Since the same holds for $(|\xi|^2-E)^{-1}$, $\widehat f$, $\widehat{\dl f}$, and $\widehat{\dl^2 f}$ have $L^2$ also analytic extensions to this domain, so $\dl^2 f$ is exponentially decaying in $L^2$.
\end{proof}

Finally, based on these lemmas, we can prove Proposition \ref{prop23}.

\begin{proof}[Proof of Proposition \ref{prop23}] With no loss of generality, let $\beta(0)=v(0)=D(0)=0$. 
Assume for now that $\beta$ is small in $\dot H^{1/2} \cap C$ and that $e^{\beta} \dot v$ and $e^{-\beta} \dot D$ are small in the $\dot H^{-1/2} \cap \partial_t^{-1} C$ norm.

Setting $z_1 = e^{iv(t)x} Z$, $f_1 = e^{iv(t)x} F$, we obtain that
$$\begin{aligned}
\partial_t Z &= e^{-ivx} \partial_t z_1 - e^{-ivx} i(\dot v \cdot x) z_1,\\
\dl Z &= -e^{-ivx}ivz_1 + e^{-ivx} \dl z_1,\\
\Delta Z &= -e^{-ivx} |v|^2 z_1 - 2e^{-ivx} iv \dl z_1 + e^{-ivx} \Delta z_1.
\end{aligned}$$
The equation translates into
$$
i \partial_t z_1 + (\dot v \cdot x) z_1 + |v|^2 z_1 + 2iv \dl z_1 + (-\Delta + e^{-\beta(x\dl+2)} V(x-\gamma(t))) z_1 = f_1.
$$
Setting $z_2 = z_1(x+\gamma(t), t)$, $f_2 = f_1(x+\gamma(t), t)$, we get
$$
\partial_t z_1 = \partial_t z_2(x-\gamma(t), t) - \dot \gamma \dl z_2(x-\gamma(t), t),
$$
so
$$
i \partial_t z_2 - i \dot \gamma \dl z_2 + (\dot v \cdot x) z_2 + |v|^2 z_2 + 2iv \dl z_2 + (-\Delta+e^{-\beta(x\dl+2)}V) z_2 = f_2,
$$
which further simplifies to
$$
i \partial_t z_2 + (\dot v \cdot x) z_2 - i \dot D \dl z_2 - |v|^2 z_2 + (-\Delta + e^{-\beta(x\dl+2)}V) z_2 = f_2.
$$
Finally, if $z_3 = e^{\beta(t)(x\dl+3/2)} z_2$, $f_3 = e^{\beta(t)(x\dl+3/2)x} f_2$, we obtain that
$$\begin{aligned}
\partial_t z_2 &= e^{-\beta(t)(x\dl+3/2)} \partial_t z_3 - e^{-\beta(t)(x\dl+3/2)} \dot \beta (x\dl+3/2) z_3,\\
\Delta z_2 &= e^{-\beta(t)(x\dl+7/2)} \Delta z_3.
\end{aligned}$$
Then
$$
i \partial_t z_3 + e^{\beta}(\dot v \cdot x) z_3 - i e^{-\beta} \dot D \dl z_3 - i \dot \beta (x\dl+3/2) - |v|^2 z_3 + e^{-2\beta}(-\Delta +V) z_3 = f_3.
$$
Let the components of the solution and of the source term in this frame~be
\begin{align*}
&\tilde Z = P_c z_2,\ Z_p(t) = \sum_{k=1}^N Z_k(t),\ Z_k(t) = \zeta_k(t) g_k = P_k z_2,\\
&\tilde F(x, t) = P_c f_2,\ F_k(x, t) = \Phi_k(t) g_k = P_k f_2.
\end{align*}
Also denote $K = e^{\beta} \dot v \cdot x - i e^{-\beta} \dot D \dl - i\dot \beta (x\dl+3/2)$. The equation for $\tilde Z$ becomes
\begin{equation}\begin{aligned}\lb{eq_pc}
&i \partial_t \tilde Z + K \tilde Z - |v|^2 \tilde Z + e^{-2\beta} H \tilde Z = \tilde F + P_p K \tilde Z - P_c K Z_p,\\
&\tilde Z(0) = P_c Z(0),
\end{aligned}\end{equation}
while the equation for each $Z_k$, $1 \leq k \leq N$, can be written as
$$\begin{aligned}
&i \partial_t Z_k + P_k K Z_p - (|v|^2 - e^{-2\beta} E_k) Z_k = F_k - P_k K \tilde Z,\\
&Z_k(0) = P_k Z(0).
\end{aligned}$$
Let $p=(p_{k\ell})$ be the $N \times N$ matrix of elements $p_{k\ell} = \langle g_k, K g_{\ell} \rangle$, $E = (E_{jk})$ 
be the diagonal matrix with the energies on the diagonal; note that $p$ and $E$ are selfadjoint. Also define the operators
$$\begin{aligned}
T \zeta &= \sum_{k=1}^N \zeta_k P_c K g_k, & T^* Z &= \big(\langle K P_c Z, g_k \rangle \big)_{1 \leq k \leq N}; \\
\tau \zeta &= \sum_{k=1}^N \zeta_k g_k, & \tau^* Z &= \big(\langle Z, g_k \rangle \big)_{1 \leq k \leq N}.
\end{aligned}$$
Finally, let
$$\begin{aligned}
\zeta(t) &= \tau^* Z_2 = (\zeta_1(t), \ldots, \zeta_N(t))^T,\\
\Phi(t) &= \tau^* F_2 = (\Phi_1(t), \ldots, \Phi_N(t))^T
\end{aligned}$$
be the column vectors having $\zeta_k(t)$ and $\Phi_k(t)$ as entries.
Setting
$$
P(t) = \int_0^t p(s) \dd s = \big\langle g_k, \big(v(t)x-iD(t)\dl-i\beta(t)(x\dl+3/2)\big) g_{\ell}\big\rangle,
$$
we obtain the system
$$
i \partial_t \zeta + (\dot P - |v|^2 - e^{-2\beta}E) \zeta = \Phi + T^* \tilde Z.
$$

Let $A(t) = e^{P(t)}$, $\tilde \zeta = A(t) \zeta$, $E(t) = A(t) E A(t)^{-1}$, $\tilde \Phi(t) = A(t) \Phi(t)$; also let $B(t)$ be a family of unitary matrices that solves $\partial_t B(t) = e^{-i|v(t)|^2+ie^{-2\beta}E(t)} B(t)$. Then
$$
i \partial_t \tilde \zeta - (|v(t)|^2-e^{-2\beta}E(t)) \tilde \zeta = \tilde \Phi + A(t) T^* \tilde Z
$$
and the solution can be written in integral form as
\be\lb{modul}
\tilde \zeta(t) = \int_{-\infty}^t B(t) B(s)^{-1} (\delta_{s=0} \tilde \zeta(0) + i \tilde \Phi(s) + i A(s) T^* \tilde Z) \dd s.
\ee

Concerning the projection on the continuous spectrum, fix $\delta>0$ and write (\ref{eq_pc}) in the equivalent form
\begin{align*}
i \partial_t \tilde Z + K \tilde Z - |v|^2 \tilde Z + e^{-2\beta}(H P_c + i \delta P_p) \tilde Z = \tilde F + \tau T^* \tilde Z - T \zeta.
\end{align*}
Note that
$$
H P_c +i \delta P_p = H_0 + V - P_p (\mc H - i \delta) = H_0 + \tilde V_1 \tilde V_2.
$$
By Lemma \ref{lemma32}, we obtain
\be\lb{2.20}
i \partial_t \tilde Z + K \tilde Z - |v|^2 \tilde Z + e^{-2\beta} (H_0 + \tilde V_1 \tilde V_2) \tilde Z = \tilde F + \tau T^* \tilde Z - T \zeta.
\ee
As a model, consider the simplified equation
$$
i \partial_t f + K f - |v|^2 f + e^{-2\beta} H_0 f = F, f(0) \text{ given}.
$$
Letting $S(t) = e^{\beta(t)(x\dl+3/2)} e^{\gamma(t) \dl} e^{iv(t)x}$ and
$$\begin{aligned}
f(x, t) &= S(t)^{-1} g(x, t) := e^{-iv(t)x} g(x-\gamma(t), t), \\
F(x, t) &= S(t)^{-1} G(x, t) := e^{-iv(t)x} G(x-\gamma(t), t),
\end{aligned}$$
this becomes
$$
i \partial_t g + H_0 g = G, g(0) = U(0)^{-1} f(0),
$$
so
$$
g = e^{it H_0} g(0) - i \int_{-\infty}^t e^{i(t-s) H_0} G(s) \dd s
$$
and
$$
f = S(t)^{-1} e^{it H_0} S(0) f(0) - i \int_{-\infty}^t S(t)^{-1} e^{i(t-s) H_0} S(s) F(s) \dd s.
$$
Returning to (\ref{2.20}), we obtain that
\be\lb{2.193}
\tilde Z = \int_{-\infty}^t S(t)^{-1} e^{i(t-s) H_0} S(s) \big((\delta_{s=0} \tilde Z(0) - i\tilde F + \tau T^* \tilde Z - T \zeta) + i\tilde V_1 \tilde V_2 \tilde Z\big) \dd s.
\ee
Denote
$$\lb{3.144}\begin{aligned}
\tilde T_{\tilde V_2, \tilde V_1} F(t) &= \int_{-\infty}^t \tilde V_2 S(t)^{-1} e^{i(t-s) {H}_0} S(s) \tilde V_1 F(s) \dd s, 
\end{aligned}$$
respectively
$$\begin{aligned}
\tilde T_{\tilde V_2, I} F(t) &= \int_{-\infty}^t \tilde V_2 S(t)^{-1} e^{i(t-s) {H}_0} S(s) F(s) \dd s.
\end{aligned}$$
Then, rewrite Duhamel's formula (\ref{2.193}) as
\be\lb{3.186}\begin{aligned}
(I - i \tilde T_{\tilde V_2, \tilde V_1}) \tilde V_2 \tilde Z(t) &= \tilde T_{\tilde V_2, I} (\delta_{t=0} \tilde Z(0) - i\tilde F + \tau T^* \tilde Z - T \zeta).
\end{aligned}\ee
We compare $\tilde T_{\tilde V_2, \tilde V_1}$ with the kernel
$$\begin{aligned}
T_{\tilde V_2, \tilde V_1} F(t) &= \int_{-\infty}^t \tilde V_2 e^{i(t-s) H_0} \tilde V_1 F(s) \dd s. 
\end{aligned}$$
The comparison takes place in the following algebra $\tilde K$:
\begin{definition}\lb{deg_kt}
$
\tilde K = \{T(t, s) \mid \sup_s \|T(t, s) f\|_{M_t L^2_x} \leq C \|f\|_2\}.
$
\end{definition}
Indeed, $I - i T_{\tilde V_2, \tilde V_1}$ is invertible in $\tilde K$, with the inverse explicitly given by
$$
\big((I - i T_{\tilde V_2, \tilde V_1})^{-1} F\big)(t) = I + i \int_{-\infty}^t \tilde V_2 e^{i(t-s) (H P_c + i \delta P_p)} \tilde V_1 F(s) \dd s
$$
and bounded due to Strichartz estimates, following \cite{bec} and earlier results. Then, we can also invert in (\ref{3.186}), if $\|\gamma\|_\CC$ is sufficiently small; indeed, $\CC \subset L^{\infty}$, so by Lemma \ref{lemma2.1}
$$
\lim_{\|D\|_{\CC}+\|v\|_{\CC} \to 0} \|\tilde T_{\tilde V_2, \tilde V_1} - T_{\tilde V_2, \tilde V_1}\|_{\B(L^2_{t, x}, L^2_{t, x})} = 0.
$$
This implies that we can invert $I - i \tilde T_{\tilde V_2, \tilde V_1}$ and obtain
$$
\tilde V_2 \tilde Z(t) = (I - i \tilde T_{\tilde V_2, \tilde V_1})^{-1} \tilde T_{\tilde V_2, I} (\delta_{t=0} \tilde Z(0) - i\tilde F + \tau T^* \tilde Z - T \zeta).
$$
Using the further notation
$$\begin{aligned}
\tilde T_{I, \tilde V_1} F(t) &= \int_{-\infty}^t S(t)^{-1} e^{i(t-s) {H}_0} S(s) \tilde V_1 F(s) \dd s,
\end{aligned}$$
and setting $\tilde T_{I, I}$ to simply be the Schr\"{o}dinger evolution operator
$$\begin{aligned}
\tilde T_{I, I} F(t) &= \int_{-\infty}^t S(t)^{-1} e^{i(t-s) {H}_0} S(s) F(s) \dd s,
\end{aligned}$$
we obtain from (\ref{3.186}) that
$$
\tilde Z = \big(\tilde T_{I, I} + \tilde T_{I, \tilde V_1} (I - i \tilde T_{\tilde V_2, \tilde V_1})^{-1} \tilde T_{\tilde V_2, I}\big) (\delta_{t=0} \tilde Z(0) - i\tilde F + \tau T^* \tilde Z - T S(t)^{-1} \tilde \zeta(t)).
$$
We consider this equation as part of a system, together with the modulation equation (\ref{modul}). Using auxiliary variables $\zeta_1$, $\zeta_2$, $Z_1$, and $Z_2$, we rewrite both equations as follows, in order to apply a fixed point argument:
\be\begin{aligned}\lb{2.19}
Z_1 =& \big(\tilde T_{I, I} + \tilde T_{I, \tilde V_1} (I - i \tilde T_{\tilde V_2, \tilde V_1})^{-1} \tilde T_{\tilde V_2, I}\big) \\
&\big(\delta_{t=0} \tilde Z(0) - i\tilde F + \tau T^* Z_2 - T A(t)^{-1} \zeta_2(t) \big) \\
\zeta_1(t) =& \int_{-\infty}^t B(t) B(s)^{-1} \big(\delta_{s=0} \tilde \zeta(0) + i \tilde \Phi(s) + A(s) T^* Z_2(s)\big) \dd s.
\end{aligned}\ee
For initial data, we take $Z_1(0) = Z_2(0) = P_c Z(0)$ and $\zeta_1(0) = \zeta_2(0) = \tilde \zeta(0)$.

Assume that $Z_2$ is in the Strichartz space $L^2_t L^{6, 2}_x$, $B(t)^{-1} A(t) T^* Z_2 \in \CC'$, and $B(t)^{-1} \zeta_2(t) \in\CC$; we prove the same for $Z_1$ and $\zeta_1$.

To begin with,
by (\ref{2.19}) and Lemma \ref{lemma24},
$$
\|B(t)^{-1} \zeta_1(t)\|_\CC \les \|Z(0)\|_2 + \|B(t)^{-1} \tilde \Phi(t)\|_{\CC'} + \|B(t)^{-1} A(t) T^* Z_2\|_{\CC'}.
$$
Note that $B(t)$, $B(t)^{-1} \in W^{1, \infty}$ and $A(t)$, $A(t)^{-1} \in \dot H^{1/2} \cap C$. Considering each matrix component separately, for $\gamma \in \dot H^{1/2} \cap C$, (\ref{eqn2.11}) implies that
$$
\|B(t)^{-1} A(t) T^* \tilde T_{I, I} F\|_\CC \les \|F\|_{L^2_t L^{6/5, 2}_x},\ \|B(t)^{-1} A(t) T^* \tilde T_{I, V_1} F\|_\CC \les \|F\|_{L^2_{t, x}}.
$$
The dual of (\ref{eqn2.11}) results in
$$\begin{aligned}
&\|\tilde T_{I, I} T A(t)^{-1} \zeta_2\|_{L^2_t L^{6, 2}_x}
 + \|\tilde T_{\tilde V_2, I} T A(t)^{-1} \zeta_2\|_{L^2_{t, x}}
  \les \\
&\les (\|\beta\|_{\dot H^{1/2} \cap C} + \|e^{\beta} \dot v\|_{\dot H^{-1/2} \cap \partial_t C}
   + \|e^{\beta} \dot D\|_{\dot H^{-1/2} \cap \partial_t C}) \|B(t)^{-1} \zeta_2\|_\CC.
\end{aligned}$$
Finally, from (\ref{eq2.45}) we infer that
$$\begin{aligned}
&\|B(t)^{-1} A(t) T^* \tilde T_{I, I} T A(t)^{-1} \zeta_2\|_\CC \les\\
&\les (\|\beta\|_{\dot H^{1/2} \cap C} + \|e^{\beta} \dot v\|_{\dot H^{-1/2} \cap \partial_t C} + \|e^{\beta} \dot D\|_{\dot H^{-1/2} \cap \partial_t C}) \|B(t)^{-1} \zeta_2\|_\CC.
\end{aligned}$$
In this estimate we also use the fact that the bound states of $-\Delta+V$ have two derivatives exponentially decaying in $L^2$. We reach the same conclusion if $\beta$, $v$, $D \in \Gamma$ and bound states have one additional derivative.

These term-by-term bounds result in the overall estimate
$$\begin{aligned}
&\|Z_1\|_{L^2_t L^{6, 2}_x} + \|B(t)^{-1} A(t) T^* Z_1\|_{\CC'} + \|B(t)^{-1} \zeta_1(t)\|_\CC \les \\
&\les (\|\beta\|_{\dot H^{1/2} \cap C} + \|e^{\beta} \dot v\|_{\dot H^{-1/2} \cap \partial_t C} + \|e^{\beta} \dot D\|_{\dot H^{-1/2} \cap \partial_t C}) \\
&(\|Z_2\|_{L^2_t L^{6, 2}_x} + \|B(t)^{-1} A(t) T^* Z_2\|_\CC + \|B(t)^{-1} \zeta_2(t)\|_\CC) + \\
&+\|Z(0)\|_2 + \|P_c F\|_{L^2_t L^{6/5, 2}_x + L^1_t L^2_x} + \|B(t)^{-1} \tilde \Phi\|_{\CC'}.
\end{aligned}$$
Taking the difference between two solutions to (\ref{2.19}) corresponding to different values of $(Z_2, \zeta_2)$, that is $(Z_2^j, \zeta_2^j) \mapsto (Z_1^j, \zeta_1^j)$, $j=1, 2$, the terms related to $Z_0$ and $F$ cancel:
$$\begin{aligned}
&\|Z_1^2-Z_1^1\|_{L^2_t L^{6, 2}_x} + \|B(t)^{-1} A(t) T^* (Z_1^2-Z_1^1)\|_\CC + \|B(t)^{-1} (\zeta_1^2(t)-\zeta_1^1(t))\|_\CC \les \\
&\les (\|\beta\|_{\dot H^{1/2} \cap C} + \|e^{\beta} \dot v\|_{\dot H^{-1/2} \cap \partial_t C} +\|e^{\beta} \dot D\|_{\dot H^{-1/2} \cap \partial_t C}) \\
&\big(\|Z_2^2-Z_2^1\|_{L^2_t L^{6, 2}_x} + \|B(t)^{-1} A(t) T^* (Z_2^2-Z_2^1)\|_\CC + \|B(t)^{-1} (\zeta_2^2(t) - \zeta_2^1(t))\|_\CC\big).
\end{aligned}$$
Therefore, the mapping $(Z_2, \zeta_2) \mapsto (Z_1, \zeta_1)$ is a contraction and its fixed point $(\tilde Z, \zeta)$ is a solution to the original system, satisfying the bound
$$\begin{aligned}
&\|\tilde Z\|_{L^2_t L^{6, 2}_x} + \|B(t)^{-1} A(t) T^* \tilde Z\|_\CC + \|B(t)^{-1} \zeta(t)\|_\CC \les \\
&\les \|Z(0)\|_2 + \|P_c F\|_{L^2_t L^{6/5, 2}_x + L^1_t L^2_x} + \|B(t)^{-1} \tilde \Phi\|_{\CC'}.
\end{aligned}$$
Letting $Z = \tilde Z + \tau^* \zeta$, we obtain a solution to the equation that satisfies the stated Strichartz estimates and is in $\dot H^{1/2}$ in time, if localized in space.

Finally, assume that $v$, $D$, and $\beta$ have finite, but not small, $\dot H^{1/2} \cap C$ norms, and are uniformly continuous; then we divide the interval $[0, \infty)$ into finitely many pieces,
$$
[0, \infty) = [0, t_1] \cup [t_1, t_2] \cup \ldots \cup [t_{N-1}, \infty),
$$
on each of which $\|e^{\beta}\|_{\Gamma} \|v(t)\|_{\dot H^{1/2}[t_1, t_2]} + \sup_{t \in [t_1, t_2]} \|e^{\beta}\|_{\Gamma} |v(t) - v(t_1)|$ is small and same goes for $D$ and $\beta$.

By a symmetry transformation we can set $\beta(t_1) = v(t_1) = D(t_1) = 0$. We use the same argument to conclude that the Strichartz estimates hold on each interval. Iterating, we obtain a bound on the whole real line, with a constant that grows exponentially with $N$.
\end{proof}

We proceed with the proof of Corollary \ref{ionization}.
\begin{proof}[Proof of Corollary \ref{ionization}] Just to simplify the computations, we take the scaling $\beta$ to be zero in the sequel (there is no qualitative change from doing~this).

Following (\ref{modul}), the mass transfer from $P_c L^2$ to $P_p L^2$ is expressed by the formula
$$
U_{cp}(t) P_c Z(0) = \int_{-\infty}^t \tau^* B(t) B(s)^{-1} A(s) T^* P_c Z(s) \dd s.
$$
Here $Z$ is a solution to the equation (\ref{1.1}) having $P_c Z(0)$ as the initial data. The mass transfer from $P_p L^2$ to $P_c L^2$ is given by
$$\begin{aligned}
U_{pc}(t) P_p Z(0) = &\big(\tilde T_{I, I} + \tilde T_{I, \tilde V_1}(I-i\tilde T_{\tilde V_2, \tilde V_1})^{-1} \tilde T_{\tilde V_2, I}\big) \\
&\big(\tau T^* P_c Z - T B(t)^{-1} \tau P_p Z).
\end{aligned}$$
When $\gamma$ is small in norm,
$$
\|B(t)^{-1} S(t) \tau U_{cp}(t) P_c Z(0)\|_{\CC} \les \|\gamma\|_{\CC} \|Z\|_{L^2_t L^{6, 2}_x} \les \|\gamma\|_{\CC} \|P_c Z(0)\|_2
$$
and
$$
\|Z\|_{L^{\infty}_t L^2_x} \les \|\gamma\|_{\CC} \|P_p Z(0)\|_2.
$$
By the Strichartz estimates of Proposition \ref{prop23}, it follows that the strong limit in the definition of the wave operator,
$$
W_+ Z(0) = \slim_{t \to \infty} e^{-it\Delta} Z(t),
$$
is $L^2$-bounded. Likewise, concerning the bound states, the limit
$$\begin{aligned}
\lim_{t \to \infty} B(t)^{-1} S(t) \tau Z(x-\gamma(t), t) = \\
= \int_{\R} B(s)^{-1} \big(\delta_{s=0} \zeta(0) + i\tilde\Phi(s) + A(s) T^* P_c Z \big)\dd s
\end{aligned}$$
exists due to Proposition \ref{prop23}.
\end{proof}

Next, we prove Theorem \ref{energy}.
\begin{proof}[Proof of Theorem \ref{energy}]
To simplify computations, we take $\beta=0$ and $v=0$ in the sequel. This does not change the proof in any significant way.

We split the energy into two terms, $E[Z(x+\gamma(t), t)] = E[P_c Z(x+\gamma(t), t)] + E[P_p Z(x+\gamma(t), t)]$. The bound state energy is bounded,~as
$$
E[P_p Z(x+\gamma(t), t)] = \langle E \tau^* Z, \tau^* Z \rangle
$$
and $\|\tau^* Z\|_{L^{\infty}_t} \les \|Z(0)\|_2$.

Next, we consider the dispersive part of the solution. Starting from
$$
i \partial_t P_c(t) Z(t) + H(t) P_c(t) Z(t) = P_c(t) F(t) + i (\partial_t P_c(t)) Z(t)
$$
and taking the gradient one has that
$$
i \partial_t \dl P_c(t) Z(t) + H(t) \dl P_c(t) Z(t) = \dl P_c(t) F(t) + i (\partial_t P_c(t)) Z(t) + (\dl V(t)) Z(t).
$$
Hence, if $\dl V \in L^{3/2, \infty}$, by Strichartz estimates for both $Z$ and $\dl Z$ --- and using inequalities of the form of Lemma \ref{lema35} to bound $i (\partial_t P_c(t)) Z(t)$ --- we get
$$
\|P_c(t) Z(x, t)\|_{L^{\infty}_t H^1_x \cap L^2_t \langle \dl \rangle^{-1} L^{6, 2}_x} \les \|Z(0)\|_{H^1}.
$$
This bounds the kinetic energy of the dispersive component $P_c Z$. Furthermore, we obtain
$$
\|P_c Z(x+\gamma(t)\|_{(\dot H^{1/2} \cap C)_t L^{6, 2}_x} \les \|Z(0)\|_{H^1},
$$
hence
$$
\big\||P_c Z(x+\gamma(t))|^2\big\|_{(\dot H^{1/2} \cap C)_t L^{3, 1}_x} \les \|Z(0)\|_{H^1}^2.
$$
This leads to a bound on potential energy:
$$
\|\langle V(x-\gamma(t)) P_c Z, P_c Z \rangle\|_{\dot H^{1/2} \cap C} \les \|Z(0)\|_{H^1}^2 \|V\|_{|\dl|^{-1} L^{3/2, \infty}} (1+\|\gamma\|_{\CC}).
$$
Note that the potential energy of the dispersive component, $E_p[P_c Z(x+\gamma(t)]$, goes to zero in the average, due to Strichartz estimates:
$$
\|\langle V(x-\gamma(t)) P_c Z, P_c Z \rangle\|_{L^1_t} \les \|V\|_{L^{3/2, \infty}} \|Z(0)\|_2^2.
$$
\end{proof}

\section*{Acknowledgments}
The authors would like to thank Herbert Koch for mentioning the work of Lyons and Pierre Germain for mentioning the Strauss exponent for this equation.

A.S.\ is partially supported by the NSF grant DMS--0903651.

\end{document}